\documentclass[a4paper,12pt]{article}    
           
\usepackage[left=1in,right=1in,top=1in,bottom=1in]{geometry}
\usepackage{amsfonts,amssymb,amsmath,amsthm,latexsym,bbm}
\usepackage{mathtools}   
\usepackage{enumerate}
\usepackage{xcolor,colortbl,color}
\usepackage{graphicx,graphics}
\usepackage{pdfpages}
\usepackage{upgreek}
\usepackage{cite}
\usepackage[british]{babel}
\usepackage{hyperref} 
\usepackage{enumitem}
\usepackage{tikz}
\usetikzlibrary{shapes.geometric}
\usetikzlibrary{decorations.markings}
\usetikzlibrary{decorations.pathreplacing}
\usepackage{caption}
\usepackage{subcaption}
\newcommand{\MR}[1]{\href{http://www.ams.org/mathscinet-getitem?mr=#1}{MR#1}}

\providecommand\phantomsection{}

\hypersetup{
    colorlinks=false,
    pdfborder={0 0 0},
}

\makeatletter
\def\namedlabel#1#2{\begingroup
    #2%
    \def\@currentlabel{#2}%
    \phantomsection\label{#1}\endgroup
}
\makeatother

\newtheorem{theorem}{Theorem}[section]
\newtheorem{corollary}[theorem]{Corollary}
\newtheorem{proposition}[theorem]{Proposition}
\newtheorem{lemma}[theorem]{Lemma}

\numberwithin{equation}{section}

\theoremstyle{definition}
\newtheorem{definition}[theorem]{Definition}
\newtheorem{examplex}[theorem]{Example}

\newenvironment{example}
  {\pushQED{\qed}\examplex}
  {\popQED\endexamplex}

\theoremstyle{remark}
\newtheorem{remark}[theorem]{Remark}
\newtheorem{remarks}[theorem]{Remarks}
\newtheorem*{remark*}{Remark}

\newcommand{\veps}{\varepsilon}
\newcommand{\ud}{\,\mathrm{d}}

\newcommand{\Ind}[1]{\mathbbm{1}_{\{#1\}} }  
\newcommand{\Indno}[1]{\mathbbm{1}_{#1} }  

\DeclareMathOperator{\Exp}{\mathbb{E}}
\let\Pr\relax
\DeclareMathOperator{\Pr}{\mathbb{P}}

\newcommand{\bb}[1]{{\mathbb #1}}

\newcommand{\bmu}{\overline{\mu}}
\newcommand{\0}{\mathbf{0}}

\newcommand{\R}{\ensuremath{\mathbb{R}}}
\newcommand{\Z}{\ensuremath{\mathbb{Z}}}
\newcommand{\ZP}{\ensuremath{\mathbb{Z}_+}}
\newcommand{\RP}{\ensuremath{\mathbb{R}_+}}
\newcommand{\N}{\ensuremath{\mathbb{N}}}
\newcommand{\Sp}{\ensuremath{\mathbb{S}}}

\newcommand{\cF}{\mathcal{F}}
\newcommand{\rc}{\textup{c}}
\newcommand{\tra}{{\scalebox{0.6}{$\top$}}}

\DeclareMathOperator{\sign}{sgn} 
\DeclareMathOperator{\trace}{tr}

\newcommand{\Rq}{\mathbb{X}}
\newcommand{\Rqint}{\Rq_\circ}
\newcommand{\Rqbx}{\Rq_1}
\newcommand{\Rqb}{\Rq_\textup{b}}
\newcommand{\Rqby}{\Rq_2}
\newcommand{\Rqc}{\Rq_\textup{c}}

\newcommand{\tRq}{\widetilde{\Rq}}
\newcommand{\tRqint}{\tRq_\circ}
\newcommand{\tRqbx}{\tRq_1}
\newcommand{\tRqb}{\tRq_\textup{b}}
\newcommand{\tRqby}{\tRq_2}

\newcommand{\bbW}{\mathbb{W}}
\newcommand{\bbWint}{\bbW_\circ}
\newcommand{\bbWbx}{\bbW_1}
\newcommand{\bbWby}{\bbW_2}
\newcommand{\bbWc}{\bbW_\textup{c}}

\newcommand{\nuint}{\nu_\circ}
\newcommand{\pint}{p_\circ}

\newcommand{\I}[1]{I_{#1}}
\newcommand{\bI}[1]{\overline{I}_{#1}}
\newcommand{\barq}{\overline{q}}
\newcommand{\tZ}{\ensuremath{\widetilde{Z}}}
\newcommand{\Ts}{\ensuremath{T_\Sigma}}
\newcommand{\bigmid}{\,\big\vert\,}
\newcommand{\Bigmid}{\,\Big\vert\,}
\newcommand{\biggmid}{\,\bigg\vert\,}
\newcommand{\aref}[1]{{\textup{\ref{#1}}}}
\newcommand{\cWs}{\mathcal{W}_\Sigma}
\newcommand{\txi}{\xi'}
\newcommand{\ttau}{\widetilde\tau}
\newcommand{\fR}{\mathfrak{R}}

\newlist{myenumi}{enumerate}{10}
\setlist[myenumi]{leftmargin=0pt, labelindent=\parindent, listparindent=\parindent, labelwidth=0pt, itemindent=!, itemsep=1pt, parsep=4pt}

\newlist{thmenumi}{enumerate}{10}
\setlist[thmenumi]{leftmargin=0pt, labelindent=\parindent, listparindent=\parindent, labelwidth=0pt, itemindent=!}

\begin{document}

\title{Passage-times for partially-homogeneous reflected random walks on the quadrant}

\author{Conrado da~Costa\footnote{\scalebox{0.88}{Department of Mathematical Sciences, Durham
  University, Upper Mountjoy Campus, Durham DH1 3LE, UK.}} \and Mikhail Menshikov\footnotemark[1] \and Andrew
  Wade\footnotemark[1]}
\date{\today}
   
\maketitle

\begin{abstract}
We consider a random walk on the first quadrant of the square lattice, whose increment law is, roughly speaking, homogeneous along a finite number of half-lines near each of the two boundaries, and hence essentially specified by finitely-many transition laws near each boundary, together with an interior transition law that applies at sufficient distance from both boundaries. Under mild assumptions, in the (most subtle) setting in which the mean drift in the interior is zero, we classify recurrence and transience and provide power-law bounds on tails of passage times; the classification depends on the interior covariance matrix, the (finitely many) drifts near the boundaries, and stationary distributions derived from two one-dimensional Markov chains associated to each of the two boundaries. As an application, we consider reflected random walks related to multidimensional variants of the Lindley process, for which the recurrence question  was studied recently by Peign\'e and Woess (\emph{Ann.\ Appl.\ Probab.}\ {\bf 31}, 2021) using different methods, but for which no previous quantitative results on passage-times appear to be known.
\end{abstract}

\medskip

\noindent
{\em Key words:}  Reflecting random walk; partial homogeneity; recurrence classification; passage-time moments; random walks in wedges; multidimensional Lindley process.

\medskip

\noindent
{\em AMS Subject Classification:} 60J10 (Primary) 60G50 (Secondary).

\tableofcontents

\section{Introduction}
\label{sec:intro}

\subsection{Overview}
\label{sec:overview}

Reflecting random walks on lattice orthants have received much attention, with motivation coming from many problems of applied probability in which coordinate processes are constrained to be non-negative, queueing theory being the primary example. 
The interactions of the process with the boundary mean that such processes are typically spatially inhomogeneous, and much work has been done to understand asymptotic behaviour under  the assumption of homogeneity when far from the boundary (in the ``interior''). The literature is extensive, and for a selection of surveys and recent work we indicate~\cite{fmm,fim,raschel,kw19,pw21,vladas,km14,ck18}; \S\ref{sec:literature} below has a  more specific discussion.

The classical case of maximal homogeneity in the quadrant, for example, supposes one increment distribution in the interior, one for the horizontal part of the boundary, and one for the vertical part (what happens at the corner is inessential). However, a consequence of such a strong restriction is that interior jumps cannot exceed one unit in length in the directions of the boundaries. Applications motivate replacing the maximal homogeneity assumption with \emph{partial homogeneity}, which permits finite-range jumps but consequently requires a finite collection of increment distributions and admits more complex behaviour.

In the maximally homogeneous setting, the classification of recurrence, transience, and stability of the reflecting random walk in the quadrant  is well understood (we discuss the literature in some more detail in \S\ref{sec:literature} below).
In the partially homogeneous setting, however,
only the case of \emph{non-zero drift} in the interior
has been addressed completely. The aim of this paper is to study the \emph{zero-drift} setting, which is more subtle, for reasons that we explain later (in the maximally homogeneous case, several decades passed between the solution of the non-zero and zero drift cases).
In particular, we establish a classification between recurrence and transience, and, in the recurrent case, provide quantitative bounds on the tails of return times to finite sets, with lower and upper bounds of matching power-law exponent.

Before explaining in full our model and stating our main result (in~\S\ref{sec:results}, below), in the next section we give an application of our result, which can be stated in isolation with minimal notation and which provides new tail bounds on passage times for two particular classes of reflecting random walks,
which are
 \emph{multidimensional Lindley processes}: see~\cite{ck18,kw19,peigne92,pw21} and references therein.

\subsection{A motivating application}
\label{sec:application}

Throughout the paper we write $\ZP:= \{0,1,2,\ldots\}$, $\N := \{1,2,3,\ldots\}$, and $\RP := [0,\infty)$.
In the present section, as well as in the more general 
setting of~\S\ref{sec:results} below,
we will study a discrete-time, time-homogeneous Markov chain with state space $\Rq := \ZP^2$. 
In the present section, the Markov chains we consider are derived from a sequence of independent random variables, together with some constraints to keep the process in $\Rq$.

Let $\zeta, \zeta_1, \zeta_2, \ldots$ be i.i.d.~random variables on $\Z^2$.
For $z=(x,y) \in \R^2$ define $|z| := (|x|, |y|) \in \ZP^2$
and $z^+ := ( \max \{x, 0\}, \max \{ y, 0 \} )$, i.e.,
the operations act componentwise. Denote by $\| \, \cdot \, \|$ the Euclidean norm on $\R^2$.
When vectors in $\R^2$ appear in formulas, they are to be interpreted as column vectors,
although, for typographical convenience, we sometimes write them as row vectors. Here are the two Markov chains we consider.

\begin{definition}[Lindley random walk]
\label{def:lindley}
The \emph{Lindley random walk}  with increment distribution~$\zeta$
is the Markov process $L = (L_n , n \in \ZP)$ on $\Rq$ defined via
\[ L_{n+1} := ( L_n + \zeta_{n+1} )^+ , \text{ for } n \in \ZP. \]
\end{definition}

\begin{definition}[Mirror-reflected random walk]
\label{def:mirror}
The \emph{mirror-reflected random walk} with increment distribution~$\zeta$ is
the Markov process $M = (M_n , n \in \ZP)$
on $\Rq$ defined via
\[ M_{n+1} := | M_n + \zeta_{n+1} | , \text{ for } n \in \ZP. \]
\end{definition}

These two models have been the subject of some interest in the last few years. 
The Lindley random walk is a multidimensional analogue of the classical \emph{Lindley equation} on $\RP$,
and has been studied in e.g.~\cite{pw21,ck18}.
The mirror-reflected random walk has been studied under the generic name ``random walk with reflections'' (see e.g.~\cite{kw19}), but   in the present paper we avoid this terminology, since the model presented in \S\ref{sec:partial_homogeneity}
is also a random walk with reflections, but of a more general type. 

Let $e_1 := (1,0)$ and $e_2 := (0,1)$ denote the orthonormal basis vectors of $\R^2$,
and 
 $\0 := (0,0)$,  the zero vector. We make the following assumption on the distribution of~$\zeta$.

\begin{description}
\item[\namedlabel{hyp:lindley-increments}{(A)}] 
Suppose that $\Exp [ \| \zeta \|^\nu ] < \infty$ for some $\nu >2$, and $\Exp \zeta = \0$. 
Suppose also that there exists $R \in \ZP$ such that $\Pr ( e_k^\tra \zeta \geq - R ) = 1$ for $k \in \{1,2\}$. Finally, suppose that
$\Sigma = \Exp [ \zeta \zeta^\tra ]$ is 
positive definite, and denote by $\rho = \frac{\Sigma_{12}}{\sqrt{\Sigma_{11} \Sigma_{22}}} \in (-1,1)$
the associated correlation coefficient.
\end{description}

 For a process $Z = (Z_n , n \in \ZP)$ taking values in $\R^2$,
 define the associated \emph{first passage time} of $Z$ into a ball of radius $r >0$ around the origin by
\begin{equation}
    \label{eq:def-passage-time-ball}
    \tau_Z (r) := \inf \{ n \in \ZP : \| Z_n \| \leq r \} .
\end{equation}
The main result of this section, Theorem~\ref{thm:lindley}, 
presents for the processes $L$ and $M$ conditions for transience and recurrence (recovering a variant of a recent result of Peign\'e \& Woess~\cite{pw21})
and also asymptotics for the passage times $\tau_L, \tau_M$, which we believe to be new.
The result will be obtained as a consequence of our more general result on partially-homogeneous random walks given in~\S\ref{sec:results} below.
Recall that the (principal branch of the) arc-cosine function $\arccos : [-1,1] \to [0,\pi]$ is given by
\begin{equation}
    \label{eq:arccos}
 \arccos \lambda := \int_\lambda^1 \frac{1}{\sqrt{1-t^2}} \ud t , \text{ for } -1 \leq \lambda \leq 1. \end{equation}

\begin{theorem}
\label{thm:lindley}
Let $\zeta$ satisfy hypothesis~\aref{hyp:lindley-increments},
and consider
either $Z=L$, 
 the Lindley random walk, 
or $Z=M$, the mirror-reflected random walk, as at Definitions~\ref{def:lindley}
and~\ref{def:mirror}, respectively. Suppose that the Markov chain $Z$ is irreducible on~$\Rq$. Then it holds that $Z$ is recurrent if $\rho > 0$ and $Z$ is transient if $\rho < 0$. Moreover, if $\rho >0$,  
then for every $r>0$ there exists $r_0>r$ such that $\tau_Z(r)$ defined by~\eqref{eq:def-passage-time-ball} satisfies the tail asymptotics 
\begin{equation}\label{eq:tail-Lindley}    
\lim_{n \to \infty} \frac{\log \Pr_z ( \tau_Z (r) > n)}{\log n } = - \left( 1 - \frac{\pi}{2 \arccos{ (- \rho)}}  \right), \text{ for all } z \in \Rq : \| z \| > r_0.
\end{equation}
\end{theorem}

\begin{remark}
    \label{rem:lindley}
For {the process} $L$, the recurrence/transience part of Theorem~\ref{thm:lindley} 
is contained in  Theorem~1.1(c) of~\cite{pw21},
whose proof combines deep results of Denisov \& Wachtel on exit times from cones~\cite{dw}
with a structural result, Lemma~1.3 of~\cite{pw21},
which relates recurrence to integrability of the time for the $\Z^2$-valued random walk with increment distribution $-\zeta$ to exit from a quadrant. 
The fact that the expected time for a random walk to exit $\RP^2$ is finite if and only if the increment correlation is negative seems to originate with Klein~Haneveld \& Pittenger~\cite{KH87,khp}: see Example~2.6.6 of~\cite[pp.~65--66]{Bluebook} for a short proof (in the context of Brownian motion, the result goes back further, to Spitzer~\cite{spitzer}). 
The recurrence/transience result for {the process} $L$ transfers to {the process}~$M$ using Lemma~5.1 of~\cite{pw21}.
The phenomenon of recurrence/transience behaviour being driven by covariance is explored in a related but different setting in~\cite{gmmw}. 
As far as we are aware, there are no previous results on the tails of return times to compact sets for either~$L$ or~$M$.
\end{remark}

\subsection{Remarks on the literature}
\label{sec:literature}

In the maximally homogeneous setting, the classification of recurrence, transience, and stability of the reflecting random walk in the quadrant 
for the case with \emph{non-zero} drift in the interior goes back to  Kingman~\cite{kingman} in the early 1960s and
Malyshev~\cite{malyshev70,malyshev72b} in the early 1970s, with subsequent refinements including~\cite{rosenkrantz,zachary}.
Generically, the classification depends on the drift vector in the interior and the two  angles produced by the drifts at the boundaries (the \emph{reflections}): see~\cite{fmm,fim} for overviews. The Foster--Lyapunov approach was
outlined by Kingman~\cite{kingman} and (apparently independently)  by Malyshev~\cite{malyshev72b}, the latter in response to a plea from Kolmogorov for a ``purely probabilistic'' alternative to the analytic approach from~\cite{malyshev70}.

The recurrence classification for the case of maximally homogeneous  reflecting random walk on~$\ZP^2$ with \emph{zero-drift} in the interior
was given in the early 1990s in~\cite{afm,fmm92}; see also~\cite{fmm}.
In this case, generically, the classification depends on the increment covariance matrix in the interior as well as the two boundary reflection angles. Passage-time moments were studied in the key work~\cite{AspIasMen96}, with refinements and extensions provided in~\cite{aiSPA,aiTPA,mp2014}.

Since the 1970s, a goal has been to relax the maximal homogeneity assumption;
Zachary~\cite{zachary}, for example,
is motivated to do so in an analysis of loss networks.
Malyshev \& Menshikov~\cite{mm79}, Fayolle \emph{et al.}~\cite{fkm,fayolle89}, and
Zachary~\cite{zachary} all consider partially homogeneous random walks on $\ZP^2$, with similar homogeneity assumptions to our model of~\S\ref{sec:results}, but in the case where the drift in the interior is \emph{non-zero}.
The non-zero drift 
case is 
simpler (see Remark~\ref{rem:moments}\ref{rem:moments-v}  below). In the case of zero-drift in the interior, the only work we know that treats the partially homogeneous setting is  Menshikov \& Petritis~\cite{MenPet04},
which is close to the present work in aims and philosophy, 
but which imposes a simplifying structural assumption (see 
Remark~\ref{rem:moments}\ref{rem:moments-iv} and \S\ref{sec:petritis} below for details).

There is some structural similarity between the present setting and the case of random walks on orthants $\ZP^d$ for $d \geq 3$, homogeneous on lower-dimensional faces of the boundary~\cite{mm79};
the classification for the $d$-dimensional problem requires knowledge of the stationary distribution for some~$(d-1)$-dimensional problems. The case $d=3$ is somewhat tractable, but $d \geq 4$ leads to daunting complexity;
again, the zero-drift case is particularly challenging, and is only partially understood, even in $d=3$.

In the context of multidimensional Lindley processes and their
relatives (see \S\ref{sec:application}), the authors \cite{peigne92,pw21,ck18} work on the continuous state space $\RP^2$, while here we consider the discrete quadrant $\ZP^2$. 
The main reason for our choice is
that 
in our more general setting for reflections (as described in \S\ref{sec:partial_homogeneity} below), our method employs what we call ``stabilization'' of effective boundary reflection angles. Here we rely on certain ratio limit theorems for non-ergodic Markov chains (see \S\ref{sec:stabilization}) and the extension to continuous state-space would introduce significant additional technicality. 
The discrete setting enables us to provide a somewhat explicit algorithm to obtain the limit reflection angles on the boundary as the average of drifts with respect to the invariant measure of an embedded chain on a finite state space. In the case of the Lindley walks, this issue is less significant, since the reflection is always orthogonal and stabilization is ``trivial'' (see Example~\ref{ex:orthogonal} below).  The continuous state-space extension of the present work is  an open problem of interest and some difficulty, although one expects a similar recurrence classification could be obtained following the ideas presented in this paper, subject to overcoming the technical challenges mentioned above.

Finally, we comment on terminology. 
We adopt ``partial homogeneity'', following Borovkov~\cite{borovkov1991} and Zachary~\cite{zachary}, for example. For the same concept, 
Fayolle \emph{et al.}~\cite{fkm,fayolle89} use the phrase 
``limited state dependency''. ``Maximal homogeneity'' is used 
in~\cite{mm79} for what we call partial homogeneity, but in~\cite{fayolle89} the same phrase is used exclusively for the case $R=1$ (in our notation), and we prefer to maintain the latter distinction.

\section{Model and main results}
\label{sec:results}

\subsection{Partially homogeneous random walk}
\label{sec:partial_homogeneity}

In this section, we will study a class of 
time-homogeneous Markov chains with state space $\Rq = \ZP^2$ and one-step transition function $p:\Rq \times \Rq \to [0,1]$. 
That is, suppose that for every initial state $z \in \Rq$ we have a probability space $(\Omega, \cF, \Pr_z)$ equipped with a process $Z = (Z_0, Z_1, \ldots)$ with $Z_n = (X_n,Y_n) \in \Rq$ for all $n \in \ZP$, such that
\begin{align}
\label{eq:def-Pz}
        \Pr_z (Z_0 = z_0, \, Z_1 = z_1, \ldots, \, Z_{n} = z_n ) = \Ind { z= z_0 } \prod_{i=0}^{n-1} p(z_{i},z_{i+1}), \text{ for all } n \in \ZP,
\end{align}
with the usual convention that an empty product is~1. For example, $Z$ might be a Markov chain with transition function~$p$ on a probability space $(\Omega, \cF, \Pr)$ with an arbitrary initial $\Pr$-distribution for $Z_0$, and then $\Pr_z ( \, \cdot \, ) = \Pr ( \, \cdot \mid Z_0 = z)$ is the probability law $\Pr$ conditioned on the initial condition being equal to $z$. 
We write $\Exp_z$ for expectation with respect to $\Pr_z$. Write $\cF_n := \sigma (Z_0, Z_1, \ldots, Z_n) \subseteq \cF$ for the $\sigma$-algebra generated by $Z_0, \ldots, Z_n$, so that~\eqref{eq:def-Pz} implies that $\Pr_z ( Z_{n+1} = w \mid \cF_n ) = p (Z_n, w)$, a.s.; in such expressions, where the initial state~$z$ plays no role, we often drop the subscript and write just~$\Pr$.

For $k \in \ZP$, define $\I{k} := \{ x \in \ZP : 0 \leq x \leq k-1 \}$
and $\bI{k} := \{ x \in \ZP : x \geq k \}$.
Our main structural assumption about the law of $Z$ is \emph{partial homogeneity}; to formulate this, we 
partition $\Rq$ into 4 regions, defined through an integer parameter $R \in \ZP$, as follows.
\begin{itemize}
    \item The \emph{interior} is $\Rqint := \bI{R} \times \bI{R}$.
    \item The \emph{horizontal boundary} is
    $\Rqbx: = \bI{R} \times \I{R}$,
    and  the \emph{vertical boundary} is $\Rqby: = \I{R} \times \bI{R}$;
    we write $\Rqb:= \Rqbx \cup \Rqby$ and call $\Rqb$ simply \emph{the boundary}.
     \item Lastly, the \emph{corner} is the finite region $\Rqc: =  \I{R} \times \I{R}$.
\end{itemize}

Define the \emph{first passage time} of $A \subseteq \Rq$ by
\begin{equation}
\label{eq:def-passage-time}
  \tau_Z(A) : = \inf \{n \in \ZP : Z_n \in A\},
\end{equation}
and let $B_r := \{ z \in \R^2 \colon \| z \| \leq r\}$ denote
the closed Euclidean ball with centre at the origin and radius $r >0$;
in the special case $A=B_r$, we write simply $\tau_Z (B_r) = \tau_Z (r)$,
to coincide with the notation at~\eqref{eq:def-passage-time-ball}.
Here are our assumptions.
\begin{description}
\item[\namedlabel{hyp:partial_homogeneity}{(H)}] 
\emph{Partial homogeneity.}
Suppose that for probability mass functions on $\Z^2$ 
denoted by 
$\pint$, $p_1 (y ; \, \cdot \, )$ ($y \in \I{R}$), and
$p_2 (x; \, \cdot \, )$ ($x \in \I{R}$),
it holds that, for all $z,w \in \Rq$,
\begin{align}
    \label{eq:partial-homogeneity}
    p (z, w ) & = \begin{cases} \pint (w-z) &\text{if } z \in \Rqint ; \\
    p_1 (y ; w-z) &\text{if } z = (x,y) \in \Rqbx; \\
    p_2 (x ; w-z) &\text{if } z = (x,y) \in \Rqby. 
    \end{cases}
\end{align}
\item[\namedlabel{hyp:irreducible}{(I)}]
\emph{Irreducibility.}
Suppose that (i)~for all $z, w \in \Rqint \cup \Rqbx$,
there exists $n \in \N$ such that $\Pr ( Z_{n} = w, \, n < \tau_Z({\Rqby \cup \Rqc}) \mid Z_0 = z) >0$,
and
(ii)~for all $z, w \in \Rqint \cup \Rqby$,
there exists $n \in \N$ such that $\Pr ( Z_{n} = w, \, n < \tau_Z ( {\Rqbx \cup \Rqc} )\mid Z_0 = z) >0$.
\item[\namedlabel{hyp:moments}{(M)}] 
\emph{Bounded moments.}
For $\pint$ as in~\aref{hyp:partial_homogeneity}, there exists $\nuint > 2$ so that $\sum_{z\in \Z^2} \| z \|^{\nuint} \pint(z) < \infty$.
Moreover, for $p_1, p_2$  as in~\aref{hyp:partial_homogeneity},
suppose that there exist $\nu_1, \nu_2 > 1$ for which $\sum_{z\in \Z^2} \| z \|^{\nu_1} p_1 (y; z) < \infty$ for all $y \in \I{R}$
and $\sum_{z\in \Z^2} \| z \|^{\nu_2} p_2 (x; z) < \infty$ for all $x \in \I{R}$. Denote $\nu := \min ( \nuint, \nu_1, \nu_2) > 1$.
\item[\namedlabel{hyp:zero_drift}{(D)}]
\emph{Zero interior drift.} For $\pint$ as in~\aref{hyp:partial_homogeneity}, we have $\sum_{z \in \Z^2} z \pint(z) = \0$.
  \end{description}

\begin{remarks}
\phantomsection
  \label{rems:bounded-jumps}
\begin{myenumi}[label=(\alph*)]
\item\label{rems:bounded-jumps-a}  Given $z = (R,R) \in \Rqint$,
  we have $p(z,z') = 0$ unless $z' = (x',y')$
  satisfies $x' \geq 0, y' \geq 0$ (else $z' \notin \Rq$).
 Hence  the partial homogeneity assumption~\aref{hyp:partial_homogeneity}
implies that jumps towards the boundaries are uniformly bounded, in the sense that $\pint$ satisfies,
for all $z = (z_1, z_2)\in \Z^2$,
  \begin{equation}
    \label{eq:bounded-jumps}
    \pint (z) =0 , \text{ whenever } z_1 \wedge z_2 < - R.
  \end{equation}
A similar conclusion holds for the boundary transition rules, $p_1$ and $p_2$. That is, for all $x,y \in I_R$ and all $z = (z_1,z_2) \in \Z^2$,
  \begin{equation}
    \label{eq:bounded-jumps-12}
    p_1(y;z) = p_2(x;z) = 0  , \text{ whenever } z_1 \wedge z_2 <-R.
  \end{equation}
\item\label{rems:bounded-jumps-b}
The  hypothesis~\aref{hyp:irreducible}
ensures that for any two states in $\Rqint \cup \Rq_k$,
there is a finite-step, positive-probability path between them that avoids the opposite boundary and the corner. A consequence is that for $Z$ all states in $\Rq \setminus \Rqc$ communicate, and hence so do all states in $\Rq \setminus B_r$ for $r > R \sqrt{2}$; it may be, however, that $Z$ is not irreducible on the whole of $\Rq$ (see \S\ref{sec:petritis} below for such an example).
\end{myenumi}
  \end{remarks}

By hypothesis~\aref{hyp:moments} there exists the $2\times 2$ real matrix 
\begin{equation}
\label{eq:sigma-def}
\Sigma := \sum_{z \in \Z^2}   \pint (z) z z^\tra.
\end{equation}
Then
$\Sigma$ is symmetric, 
non-negative definite, and, by hypothesis~\aref{hyp:partial_homogeneity}, satisfies
\begin{equation}
\label{eq:interior-covariance}
\Exp_z [ (Z_1 - Z_0) (Z_1 - Z_0)^\tra ] = \Sigma, \text{ for all } z \in \Rqint. 
\end{equation}

\begin{description}
\item[\namedlabel{hyp:full_rank}{($\Sigma$)}]
We assume that the matrix $\Sigma$ is positive definite, i.e., that $\det \Sigma > 0$.
\end{description}

\begin{remark}
    Note that~\aref{hyp:full_rank} is not a consequence of~\aref{hyp:irreducible}; one can construct, for example, walks that take jumps only $(-1,+1)$ and $(+1,-1)$ in $\Rqint$ but which nevertheless satisfy~\aref{hyp:irreducible}.
\end{remark}

Assuming~\aref{hyp:irreducible}, all states in $\Rq \setminus B_r$ for $r > R \sqrt{2}$ communicate
(see Remark~\ref{rems:bounded-jumps}\ref{rems:bounded-jumps-b}). Hence it holds that either
$\Pr_z ( \tau_Z (B_r) < \infty ) =1$ for all $r >R \sqrt{2}$ and every $z \in \Rq$
(in this case we say $Z$ is \emph{recurrent})
or 
$\Pr_z ( \tau_Z (B_r) = \infty ) > 0$ for all $r >R \sqrt{2}$ and every $z \in \Rq \setminus B_r$
(in which case $Z$ is \emph{transient}).
Moreover, for $\gamma \in \RP$
it is the case that either $\Exp_z [ \tau_Z (r)^\gamma ] < \infty$ for  all $r >R \sqrt{2}$  and every $z\in \Rq$, or else 
$\Exp_z [ \tau_Z(r)^\gamma ] = \infty$ for all $r >R \sqrt{2}$ and every $z \in \Rq \setminus B_r$. The aim of this paper is to provide classification criteria for these qualitative and quantitative descriptions of the asymptotic behaviour of passage times.

\subsection{Classification theorem and tail asymptotics}
\label{sec:theorem}

Here is our main result for the partially-homogeneous random walk on the quadrant,
which classifies recurrence and transience, and, moreover, quantifies
recurrence in terms of tails of moments of hitting times;
recall the definition of $\tau_Z(r)$ from~\eqref{eq:def-passage-time-ball}.

\begin{theorem}
\label{thm:main-theorem}
Suppose that~\aref{hyp:irreducible}, \aref{hyp:partial_homogeneity}, \aref{hyp:moments},
\aref{hyp:zero_drift}, and~\aref{hyp:full_rank} hold,
and recall the definition of parameter $\nu >1$
from~\aref{hyp:moments}.
Then there exists a characteristic parameter~$\chi \in \R$,
defined via~\eqref{eq:chi-def} below, for which the following hold.
\begin{enumerate}[label=(\roman*)]
\item
\label{thm:main-theorem-i}
If $\chi >0$ then the random walk $Z$ is recurrent
  and if $\chi <0$ then   $Z$ is transient.
\item
\label{thm:main-theorem-ii}
Suppose that $\chi >0$. Then, for
all $r>R\sqrt{2}$, $\veps >0$, 
and every $z \in \Rq$,  
there exists $C := C(r,\veps,z) < \infty$ for which 
\begin{equation}
  \label{eq:upper-tail-bound}
  \Pr_z ( \tau_{Z} (r) > n ) \leq C n^{-\frac{\nu \wedge \chi}{2}+\veps},
  \text{ for all } n \in \ZP.
\end{equation}
\item  \label{thm:main-theorem-iii}
Suppose that $\nu > 2 \vee \chi $.
Then, for
all $r>0$ there exists $r_0 > r$ such that, for every $\veps >0$,
there exists $c  > 0$ so that, 
for all $z \in \Rq \setminus B_{r_0}$,
\begin{equation}
  \label{eq:lower-tail-bound}
  \Pr_z (  \tau_{Z} (r) > n ) \geq c n^{-\frac{\chi}{2}-\veps},
  \text{ for all } n \in \ZP.
\end{equation}
  \end{enumerate}
\end{theorem}

\begin{remarks}
\phantomsection
\label{rem:moments}
\begin{myenumi}[label=(\alph*)]
\item 
\label{rem:moments-i}
The parameter $\chi$
is a function of the interior covariance matrix $\Sigma$
from~\eqref{eq:sigma-def} and the angles subtended at the boundaries
by two \emph{effective boundary drift vectors},
$\bmu_1$ and $\bmu_2$ defined at~\eqref{eq:mu-def} below. A full description follows these remarks.
\item 
\label{rem:moments-ii}
Compactly, we can summarize~\eqref{eq:upper-tail-bound} and~\eqref{eq:lower-tail-bound} through the log-asymptotic
\[ \lim_{n \to \infty} \frac{\log \Pr_z ( \tau_Z(r) > n )}{\log n} = - \frac{\chi}{2} , \]
which holds for every $r>R\sqrt{2}$
and every $z \in \Rq \setminus B_{r_0}$ for appropriate $r_0 >r$, provided that the increment moment bound~\aref{hyp:moments} holds with $\nu > 2 \vee \chi$.
In terms of the passage-time moments defined through~\eqref{eq:def-passage-time}, 
it follows from~\eqref{eq:upper-tail-bound}  that for every  $A \subseteq \Rq$ 
that contains at least one element of $\Rq \setminus \Rqc$, 
and all $z \in \Rq$,
  \begin{equation}
    \label{eq:moment-finite}
    \Exp_z[ \tau_Z(A)^\gamma ] < \infty, \text{ for all }  \gamma \in \bigl(0, \tfrac{\nu \wedge \chi}{2} \bigr).
  \end{equation}
  On the other hand,  provided $\nu > 2 \vee \chi$, 
  it follows from~\eqref{eq:lower-tail-bound} that for
every finite $A \subset \Rq$ and all $z \in \Rq$ for which $\inf_{y \in A} \| z - y \|$ is large enough, 
 \begin{equation}
    \label{eq:moment-infinite}
    \Exp_z[ \tau_Z(A)^\gamma ] = \infty, \text{ for all }  \gamma > \tfrac{\chi}{2} .
  \end{equation}
\item 
\label{rem:moments-iii}
If in addition to~\aref{hyp:irreducible} one imposes the condition that there exists $\veps>0$ such that 
\begin{equation}
    \label{eq:ue}
\inf_{x \in \R^d} \Pr_z ( x^\tra ( Z_1- Z_0 ) \geq \veps \| x \|) \geq \veps \text{ for all } z \in \Rq, 
\end{equation}
then one can replace $r_0$ by $r$ in
Theorem~\ref{thm:main-theorem}\ref{thm:main-theorem-iii}, since there is uniformly positive probability to go from $z \in \Rq \setminus B_r$
to $\Rq \setminus B_{r_0}$ without visiting $B_r$. Condition~\eqref{eq:ue} is a form of \emph{uniform ellipticity},
and it is implied for $z \in \Rqint$ by~\aref{hyp:zero_drift} and~\aref{hyp:full_rank}, but the hypotheses of Theorem~\ref{thm:main-theorem} do not guarantee that~\eqref{eq:ue} holds for \emph{all} $z \in \Rq$.
\item 
  \label{rem:moments-iv}
Menshikov \& Petritis (MP)~\cite[p.~148]{MenPet04}  
obtain a version of Theorem~\ref{thm:main-theorem}
under an additional \emph{left-continuity} hypothesis that jumps are   at most unit size in the south and west directions. (The result of~\cite{MenPet04} is formally weaker than ours, as it establishes non-existence of moments, as at~\eqref{eq:moment-infinite}, but not the lower tail bound~\eqref{eq:lower-tail-bound}.)
The proof of~\cite{MenPet04} has several common elements with ours, as we describe at the appropriate points below,
but takes an algebraic approach to constructing Lyapunov functions, using a Fredholm alternative;
here we take a probabilistic approach, based on the mixing of the process near the boundary and what we call \emph{stabilization} of the effective boundary drifts (see \S\ref{sec:stabilization} for details). 
MP~remark~\cite[p.~143]{MenPet04} that the left-continuity assumption ``is important and cannot be relaxed'' in their approach;
indeed, it provides a structural simplification, as we explain in \S\ref{sec:petritis} below, along with further details of the model of~\cite{MenPet04}.
\item \label{rem:moments-v}
As mentioned in \S\ref{sec:literature}, the papers~\cite{mm79,fkm,fayolle89,zachary} all consider partially homogeneous random walks on $\ZP^2$, with similar homogeneity assumptions to~\S\ref{sec:partial_homogeneity}, but in the case where the interior drift is \emph{non-zero}.
The non-zero drift case is simpler than the zero-drift case
treated in the present paper and by MP~\cite{MenPet04}
in two important respects. First,
the interior covariance plays no part in the classification, and second, that the boundary drifts only play a role  when the corresponding one-dimensional process is ergodic, i.e., the interior drift has a negative component in the corresponding direction. 
The technical implications of these two differences are that in the non-zero drift setting, firstly, one
can use simpler Lyapunov functions (e.g., ``almost-linear'' functions, to use Malyshev's terminology)
than those based on the harmonic functions from~\S\ref{sec:harmonic},
and, secondly, that the 
necessary mixing that ``homogenizes'' the non-homogeneous boundary behaviour (through, e.g.,
the stabilization property) 
is provided by the ordinary ergodic theorem, rather than the (deeper) ratio theorem for null-recurrent processes that we use here (see \S\ref{sec:stabilization}). 
\item The case $\chi=0$ is not covered by
Theorem~\ref{thm:main-theorem} (nor is it covered in the restricted setting of~\cite{MenPet04}). The question of recurrence when $\chi=0$ is subtle, and seems to require stronger hypotheses to settle one way or the other. In the particular case of the maximally homogeneous walk studied in \cite{afm} it is shown, using $\log \log $ of quadratic functions,  that, in the critical case, corresponding to our $\chi = 0$, the process is recurrent. In particular, the return times $\tau$ are finite almost surely, but $\Exp[\tau^\gamma] = \infty$ for all $\gamma>0$: see~\cite[Theorem 2.1]{afm}. 
Moreover, in the case of the two-dimensional Lindley walk, under certain assumptions that put the process into the $\chi=0$ class, recurrence is also established: see Corollary 4.13 of~\cite{pw21}.
The above two cases show that recurrence is the natural property to expect when $\chi=0$. However, 
any recurrence is extremely finely balanced, and  in similar critical situations, second-order terms can lead either to transience or recurrence. The technical obstacle in the present context is that the stabilization of the boundary reflections does not provide sufficient rate of convergence to conclude either way. 
  \end{myenumi}
\end{remarks}

There are three main components that define the characteristic parameter $\chi$ at the heart of Theorem~\ref{thm:main-theorem}: first, a linear transformation $\Ts$ of the
quadrant associated with the covariance matrix $\Sigma$;  second, two
probability measures $\pi_1, \pi_2$ that describe the relative occupation
distributions of the process within each of the two boundaries; and, third, 
the post-transformation reflection angles $\varphi_1, \varphi_2$ between the normal to the respective boundary and the drift vector averaged according to the relevant $\pi_1, \pi_2$. 

First we explain the origin of the $\pi_k$. To do so,
 define vertical and horizontal projections of $p$, 
$q_1, q_2 : \ZP \times \ZP \to [0,1]$,  by setting for $i, j \in \ZP$,
\begin{equation}
\label{eq:q1-def}
    q_1( i , j ) := \begin{cases}
  \sum_{x \in \Z} p_1 (i ; (x, j ) ) & \text{if } i \in \I{R}, \\
  \sum_{x \in \Z} \pint ( (x, j-i) ) & \text{if } i \in \bI{R}, \end{cases}
\end{equation}
  and 
\begin{equation}
\label{eq:q2-def}
q_2( i , j ) := 
\begin{cases} 
\sum_{y \in \Z} p_2 (i ; (j,y) ) & \text{if } i \in \I{R}, \\
  \sum_{y \in \Z} \pint ( (j-i,y) ) & \text{if } i \in \bI{R} . 
\end{cases} 
\end{equation}
Note that $\sum_{j \in \ZP} q_1 (i , j) =1$ for all $i \in \ZP$ and $\sum_{j \in \ZP} q_2 (i,j) = 1$ for all $i \in \ZP$.

\begin{proposition}
\phantomsection
\label{prop:projection-chains} 
\begin{enumerate}[label=(\roman*)]
\item
\label{prop:projection-chains-i}
The transition functions $q_1$ and $q_2$ define irreducible Markov chains on $\ZP$.
\item
\label{prop:projection-chains-ii}
There exist unique positive measures $\pi_1, \pi_2$ on~$\ZP$ that 
have $\pi_1 ( \I{R} ) = \pi_2 (\I{R}) =1$ and
are invariant under $q_1$, $q_2$, respectively, i.e.,
for all $j \in \ZP$,
\begin{equation}
\label{eq:invariant-pi-def}
\sum_{i \in \ZP} \pi_1 (i) q_1(i,j) = \pi_1 (j) , \text{ and } \sum_{i \in \ZP} \pi_2 (i) q_2(i,j) = \pi_2 (j).
\end{equation}
\item
\label{prop:projection-chains-iii}
Define 
\begin{equation}\label{eq:def-eta_k}
\eta_{1}: = \inf\{ n \in \ZP :  X_n < R\}, \text{ and } \eta_{2}: = \inf\{ n \in \ZP :  Y_n < R\}. 
\end{equation}
 Then, before time~$\eta_1$,
$Y$ evolves as a Markov chain with transition function~$q_1$, i.e.,
\begin{equation}
\label{eq:q1-transition}
\Pr  ( Y_{n+1} = y \mid \cF_n ) = q_1 ( Y_n , y ), \text{ on } \{ n < \eta_{1} \}.    
\end{equation}
Similarly,
\begin{equation}
\label{eq:q2-transition}
\Pr  ( X_{n+1} = x \mid \cF_n ) = q_2 ( X_n , x ), \text{ on } \{ n < \eta_{2} \}.
\end{equation}
\end{enumerate}
\end{proposition}

Define $\mu_1, \mu_2 : \I{R} \to \R^2$, giving the drift vectors in $\Rqbx, \Rqby$ respectively, by
\begin{equation}
  \label{eq:mu-function}
  \mu_1 (i) := \sum_{z\in\Z^2} z p_1 (i ; z), \text{ and } \mu_2 (i) := \sum_{z\in\Z^2} z p_2 (i ; z) , \text{ for } i \in \I{R},
\end{equation}
and averages of these drifts according to $\pi_1, \pi_2$ given in  Proposition~\ref{prop:projection-chains}\ref{prop:projection-chains-ii} by
 \begin{equation}
  \label{eq:mu-def} \bmu_1 := \sum_{i\in\I{R}} \pi_1 (i) \mu_1 (i),  \text{ and } \bmu_2  := \sum_{i\in\I{R}} \pi_2 (i) \mu_2 (i).
\end{equation}

Under hypothesis~\aref{hyp:full_rank}, the positive-definite, symmetric matrix $\Sigma$ has a (unique) positive-definite, symmetric square root $\Sigma^{1/2}$ with an inverse $\Sigma^{-1/2}$. For any orthogonal
matrix $O$, the $2 \times 2$ real matrix $T = O \Sigma^{-1/2}$ satisfies $T \Sigma T^\tra = I$ (the identity).
If we stipulate that $T e_1$ is in the $e_1$ direction, and $T e_2$ has positive inner product with $e_2$,
this fixes~$O$ and gives us a unique $T = \Ts$ (say) such that $\Ts \Sigma \Ts^\tra = I$ 
and
\begin{equation}
    \label{eq:Ts-choices}
    e_1^\tra \Ts e_1 > 0, ~~ e_2^\tra \Ts e_2 > 0, ~ \text{and}~ e_2^\tra \Ts e_1 = 0 ;
\end{equation}
see~\eqref{eq:linearmap-explicit} for an explicit formula for~$\Ts$.
We also use $\Ts$ to denote the associated linear transformation $x \mapsto \Ts x$ on $\R^2$. Define the transformed process $\tZ$ by $\tZ_n := \Ts Z_n$, $n\in\ZP$. Then, for $z \in \Rqint$, by linearity,
the zero-drift assumption~\aref{hyp:zero_drift} implies that 
\begin{equation}
\label{eq:transformed-drift}
\Exp_z [ \tZ_{1} - \tZ_0 ]  = \Ts \Exp_z [ Z_{1} -  Z_0 ] = \0, \text{ for all } z \in \Rqint, \end{equation}
(where we emphasize that $\Exp_z$ refers still to the initial condition $Z_0 = z$). Also,
\begin{equation}
\label{eq:transformed-covariance} 
\Exp_z \bigl[ (\tZ_{1} - \tZ_0) (\tZ_{1} - \tZ_0)^\tra \bigr] = \Ts \Exp_z \bigl[ (Z_{1} -  Z_0) (Z_{1} -  Z_0)^\tra \bigr] \Ts^\tra = I, \text{ for all } z \in \Rqint,
\end{equation}
by~\eqref{eq:interior-covariance} and choice of $\Ts$.
To define angles let us denote by
$R_\theta$ the anti-clockwise rotation around the origin of $\R^2$ by angle~$\theta$, i.e., let
$R_\theta :\R^2 \to \R^2$ be defined by
\begin{equation}
  \label{eq:rotation-def}
R_\theta(x,y) := (x \cos \theta -y \sin \theta, x \sin \theta + y
\cos \theta).  
\end{equation}
Recall from \S\ref{sec:application}
that $e_1 = (1,0)$ and $e_2=(0,1)$ denote the canonical orthonormal basis vectors of $\R^2$,
and $\0$ is the origin.
Let $\Sp := \{ z \in \R^2 : \| z \| =1\}$.
Given $z,w\in \R^2 \setminus \{\0\}$, we say that $\theta \in (-2\pi,2\pi)$ is the (oriented) angle between $w$ and $z$  if 
\begin{equation}
\label{eq:angle-def}
    \frac{R_\theta(w)}{\|w\|} =  \frac{z}{\|z\|}.
\end{equation}
Note that if $\theta$ is the angle between $z,w$, then $-\theta$ is the angle between $w$ and $z$. 
Define 
\begin{itemize}
    \item $\varphi_0  \in (0,\pi)$, the angle between $\Ts e_1$ and $\Ts e_2$.
    \item $\varphi_1 \in (-\pi/2,\pi/2)$, the angle between $R_{\pi/2}(\Ts e_1)$ and $\Ts \bmu_1$.
    \item $\varphi_2 \in (-\pi/2,\pi/2)$, the angle between  $\Ts \bmu_2$ and $R_{-\pi/2}(\Ts e_2)$.
\end{itemize}
See Figure~\ref{fig:transform-angles} for an illustration.
With this notation in hand, we
define  the characteristic parameter $\chi \in \R$ introduced in Theorem~\ref{thm:main-theorem} by
\begin{equation}
  \label{eq:chi-def}
  \chi : = \frac{\varphi_1 + \varphi_2}{\varphi_0}.
\end{equation}

\begin{figure}[!ht]
\begin{center}
\scalebox{0.865}{
\begin{tikzpicture}[domain=0:10, scale = 1.0]
\filldraw (0,0) circle (1.6pt);
\draw[black,->,>=stealth,line width =.4mm] (0,0) -- (5,0);
\draw[black,->,>=stealth,line width =.4mm] (0,0) -- (0,5);
\draw[black,-,>=stealth,dashed,line width =.3mm] (0,2.5) -- (1.4,2.5);
\draw[black,-,>=stealth,dashed,line width =.3mm] (2.5,0) -- (2.5,1.4);
\draw[black,->,>=stealth,dotted,line width =.3mm] (0,2.5) -- (1.264911,2.5+0.6324555);
\draw[black,->,>=stealth,dotted,line width =.3mm] (2.5,0) -- (1.5,1.0);
\node at (4,4.5) {$\Sigma = {\scriptstyle\begin{pmatrix} 3 & -1\\-1 & 3\end{pmatrix}}$};
\node at (-0.25,0) {$\0$};
\node at (5.4,0) {$e_1$};
\node at (0,5.4) {$e_2$};
  \draw[-] (2.5,1) arc (90:135:1);
   \node at (2.24,1.3) {$\theta_1$};
     \draw (1,2.5) arc (0:26:1);
       \node at (0.6,2.1) {$\theta_2$};
\draw[rotate around={45:(3,3)}] (3,3) circle [x radius=0.5, y radius=1];
        \node at (1.5,1.3) {$\bmu_1$};
                \node at (1.2,3.4) {$\bmu_2$};
\draw[black,->,>=stealth,line width =.3mm] (6,2.5) -- (8,2.5);
    \node at (7,1.6) {$\Ts = {\scriptstyle\begin{pmatrix} \sqrt{2}/4 & \sqrt{2}/12 \\ 0 & 1/3 \end{pmatrix}}$};
\end{tikzpicture}}
\hfill
\scalebox{0.865}{
\begin{tikzpicture}[domain=0:10, scale = 1.0]
\filldraw (0,0) circle (1.6pt);
\draw[black,->,>=stealth,line width =.4mm] (0,0) -- (1.76776695,5);
\draw[black,->,>=stealth,line width =.4mm] (0,0) -- (5,0);
\draw[black,-,>=stealth,dashed,line width =.3mm] (0.88388,2.5) -- (0.88388+1.319932658,2.5-0.46666667);
\draw[black,-,>=stealth,dashed,line width =.3mm] (2.5,0) -- (2.5,1.4);
\node at (4.5,4.4) {$T_\Sigma \Sigma T_\Sigma^\tra = I$};
\node at (-0.25,0) {$\0$};
\node at (5.6,0) {$T_\Sigma e_1$};
\node at (2,5.4) {$T_\Sigma e_2$};
\draw[rotate around={45:(3+0.88388,3)}] (3+0.88388,3) circle [x radius=0.7, y radius=0.7];
    \draw (1,0) arc (0:69:1);
    \node at (0.6,0.4) {$\varphi_0$};
    \draw[black,->,>=stealth,dotted,line width =.3mm] (2.5,0) -- (1.92,1.2);
    \draw[-] (2.5,1) arc (90:115:1);
      \draw[black,->,>=stealth,dotted,line width =.3mm] (0.88388,2.5) --(2.5,2.5);
  \draw (2.0,2.5) arc (0:-20:1.1);
      \node at (1.3,2.0) {$\varphi_2$};
          \node at (2.24,1.3) {$\varphi_1$};
                  \node at (1.6,1.3) {$v_1$};
                \node at (2.4,3.0) {$v_2$};
\end{tikzpicture}}
\caption{\label{fig:transform-angles} The 
quadrant (\emph{left}) and its image under the 
transformation $T_\Sigma$ (\emph{right}) as a wedge with angular span~$\varphi_0$ defined at~\eqref{eq:inangle}. 
Also depicted are the inwards-pointing normal vectors to the boundaries (dashed lines), the effective boundary drifts (dotted lines, $\bmu_1, \bmu_2$), their transformed counterparts ($v_1, v_2$),
and the angles of the latter relative to the appropriate normal vectors $(\varphi_1, \varphi_2)$;
the sign conventions are such that in the case indicated in the figure, $\varphi_1 >0$ and $\varphi_2 < 0$.
In this example $\sigma_1^2 =\sigma_2^2 =3$, $\kappa =-1$, $\rho = -1/3$, $\bmu_1 = (-1, 1)$, and $\bmu_2 = (2,1)$ (see \S\ref{sec:linear_transformation} for definitions of $\sigma_1, \sigma_2, \rho$ in terms of entries in $\Sigma$).
The ellipse on the left represents the correlation structure of $\Sigma$ (its shape is 
 the locus of $x \mapsto \Sigma x$ over $x \in \Sp$); the circle on the right corresponds to the identity matrix.
In the contrary case $\rho > 0$, the wedge angle $\varphi_0$ is obtuse.}
\end{center}
\end{figure}
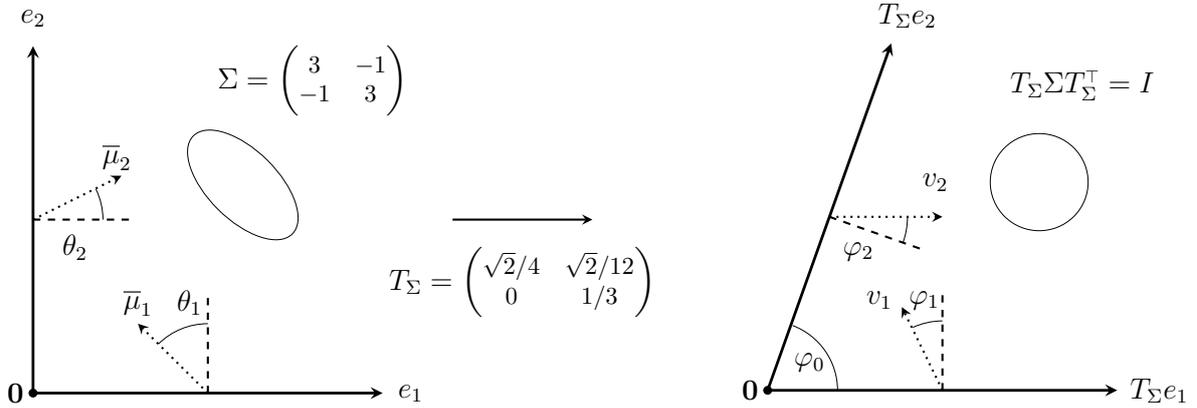 

\subsection{Organization of the rest of the paper}
\label{sec:organization}

The rest of the paper is organized around three main elements of the proof of Theorem~\ref{thm:main-theorem}.
\begin{itemize}
    \item We consider the transformation of the process under $\Ts$ introduced at~\eqref{eq:Ts-choices}, in order to obtain a process with canonical interior covariance, which is more convenient later. This results in transforming the quadrant into a wedge, and it is here that the angles that define~$\chi$ via~\eqref{eq:chi-def} are exhibited. We describe this in more detail, together with some more explicit formulas than those given in~\S\ref{sec:theorem}, in~\S\ref{sec:linear_transformation}.
    \item The contribution of the boundary behaviour to Theorem~\ref{thm:main-theorem} is via the reflection angles $\varphi_1, \varphi_2$ that enter into~$\chi$ at~\eqref{eq:chi-def}. These angles are defined in terms of $\Ts$ and $\bmu_1, \bmu_2$ from~\eqref{eq:mu-def}, and arise from a form of homogenization over the boundary: we describe and prove the stabilization result that captures this precisely in~\S\ref{sec:stabilization}.
    \item The bulk of the proof is then using Foster--Lyapunov results to deduce recurrence, transience, and bounds on tails of passage times. These arguments, concluding in the proof of Theorem~\ref{thm:main-theorem},  are presented in detail in~\S\ref{sec:proofs}.
\end{itemize}
After the proof of Theorem~\ref{thm:main-theorem} is concluded,
we present in \S\ref{sec:examples} the proof of Theorem~\ref{thm:lindley}
and explain in detail how the prior work of Menshikov \& Petritis~\cite{MenPet04}
fits as a special case of Theorem~\ref{thm:main-theorem} (see Remark~\ref{rem:moments}\ref{rem:moments-iv}).
Finally, in a short appendix we present the Foster--Lyapunov tools that we apply (\S\ref{sec:foster-lyapunov}).

Finally, before embarking on the proofs, 
we emphasize that we have chosen to present the arguments in detail, even when some technicalities are similar to those from the prototype work on passage-time moments in wedges in~\cite{AspIasMen96} (we give appropriate references at such points in the text). In part we do so for the convenience of keeping~\S\ref{sec:proofs}  self contained, but also because the exposition in~\cite{AspIasMen96} is in some places a little terse; the fact that the relevance of~\cite{AspIasMen96} 
to modern work on Lindley-type processes does not seem to have been easily appreciated, for example, suggests to us that there is value in an exposition that gives both details and intuition, while attempting to remain of a readable length.

\section{Covariance, angles, and linear transformation}
\label{sec:linear_transformation}

 This section puts some explicit geometrical formulae to the quantities
 needed to define the key parameter~$\chi$ from~\eqref{eq:chi-def}.
First, recall the definition of the interior covariance matrix $\Sigma$
from~\eqref{eq:sigma-def}. 
 Representing $Z_n = (X_n,Y_n)$ in Cartesian coordinates, write
\begin{equation}
\label{eq:sigma-entries}
  \Sigma =  
  \begin{pmatrix}
    \sigma_1^2 & \kappa\\
    \kappa & \sigma_2^2
  \end{pmatrix}.
\end{equation}
 That is, 
 for $z \in \Rqint$,
 $\sigma_1^2 = \Exp_z [ (X_{1} -X_0 )^2 ]$,
 $\sigma_2^2 = \Exp_z [ (Y_{1} -Y_0 )^2 ]$,
 and
 $\kappa = \Exp_z [ (X_{1} - X_0) (Y_{1} -Y_0 ) ]$.
Define also the \emph{correlation coefficient}
\begin{equation}
    \label{eq:rho-def}
    \rho := \rho (\Sigma) := \frac{\kappa}{\sqrt{\sigma_1^2 \sigma_2^2} }.
\end{equation}
 We choose $\Ts: \R^2\to \R^2$ to be the
linear transformation given by
\begin{equation}
  \label{eq:linearmap-explicit}
  \Ts: = \frac{1}{s \sigma_2}
  \begin{pmatrix}
    \sigma_2^2 & - \kappa\\
    0 & s
  \end{pmatrix}, \text{ where } s := \sqrt{ \det \Sigma } = \sqrt{\sigma_1^2 \sigma_2^2 - \kappa^2};
\end{equation}
note
that, by hypothesis~\aref{hyp:full_rank},   it is assumed
that $\sigma_1, \sigma_2$, and $s$ are all strictly positive,
and that $\rho$ defined at~\eqref{eq:rho-def} satisfies $\rho \in (-1,1)$.
Then $\Ts$ given by~\eqref{eq:linearmap-explicit} satisfies $\Ts \Sigma \Ts^\tra = I$
and~\eqref{eq:Ts-choices}, as we demanded in~\S\ref{sec:theorem}.
We write $\tRq := \Ts ( \Rq)$
for the image of the state space under the transformation~$\Ts$,
and similarly $\tRqint := \Ts ( \Rqint)$ and
$\tRq_k := \Ts ( \Rq_k)$.

The probabilistic significance of $\Ts$ at~\eqref{eq:linearmap-explicit} comes from the 
transformations~\eqref{eq:transformed-drift} and~\eqref{eq:transformed-covariance} 
of the interior drift and increment covariance.
From~\eqref{eq:linearmap-explicit}, 
the basis vectors $e_1, e_2$ transform via 
$\Ts(e_1) = e_1 \sigma_2/s$ and $\Ts(e_2) =(-\kappa,s)/(s\sigma_2)$ with $s>0$.
Hence the quadrant $\RP^2$ maps to the wedge $\cWs := \Ts(\RP^2)$ with angular span $\varphi_0 \in (0, \pi)$ given by
\begin{equation}
  \label{eq:inangle}
  \varphi_0 : = \arccos \left( - \frac{\kappa}{\sqrt{\kappa^2 + s^2}} \right) = \arccos (- \rho),
\end{equation}
where $\arccos$ is defined at~\eqref{eq:arccos} and $\rho$ is given by~\eqref{eq:rho-def}.

We define the directions of the transformed 
effective boundary drifts by
\begin{equation}
  \label{eq:t_drift}
  v_k := \frac{\Ts(\bmu_k)}{\| \Ts (\bmu_k) \|}, \text{ for } k  \in \{1,2\},
\end{equation}
where $\bmu_1, \bmu_2$ are defined in~\eqref{eq:mu-def}.
Denote the unit normal vectors at the  
boundaries of~$\Ts(\RP^2)$  by $n_1 := e_2$ (the normal at the horizontal boundary) and $n_2 := (\sin
\varphi_0, - \cos \varphi_0)$. Recall the definition of the rotation $R_{\varphi}$ from~\eqref{eq:rotation-def}. 
For $k \in \{1,2\}$, the 
vectors $v_k$ at~\eqref{eq:t_drift} define reflection angles $\varphi_k \in (-\pi/2, \pi/2)$ by the
  relations
\begin{equation}
  \label{eq:transformed-angles}
  v_1 =: R_{\varphi_1} n_1, \text{ and }   v_2 =: R_{-\varphi_2} n_2.
\end{equation}
In Cartesian coordinates,  
\begin{equation}
  \label{eq:boundary-drift-direction}
    v_1  =  (-\sin \varphi_1, \cos \varphi_1 ), \text{ and } 
    v_2 = ( \sin(\varphi_0 - \varphi_2), - \cos(\varphi_0 - \varphi_2)).
  \end{equation}
One can then obtain explicit 
formulae
for $\varphi_1$ and $\varphi_2$
in terms of the entries of~$\Sigma$ from~\eqref{eq:sigma-entries},
the angular span~$\varphi_0$ of the wedge from~\eqref{eq:inangle},
and angles $\theta_1, \theta_2\in (-\pi/2,\pi/2)$ such that
\begin{equation}
    \label{eq:theta-defs}
\bmu_1 = \|\bmu_1\| \big(-\sin \theta_1, \cos \theta_1  \big), \text{ and }
\bmu_2 = \|\bmu_2\| \big(\cos \theta_2, -\sin  \theta_2  \big).
\end{equation}
Then  expressions for  $\varphi_1$ and $\varphi_2$ can be given by
\begin{align}
  \label{eq:normal-angles-transformed-1}
    \varphi_1 &= \arctan \bigg(\frac{\sigma_2^2 \sin \theta_1 + \kappa \cos \theta_1}{s \cos \theta_1}\bigg),\\[.3cm]
      \label{eq:normal-angles-transformed-2}
    \varphi_2 &= \arctan\bigg( - \frac{(\sigma_2^2 \cos \theta_2 + \kappa \sin \theta_2) \cos \varphi_0 - s \sin \theta_2 \sin\varphi_0}{(\sigma_2^2 \cos \theta_2 + \kappa \sin \theta_2) \sin \varphi_0 + s \sin \theta_2 \cos\varphi_0}  \bigg),
\end{align} 
where
$\arctan: \R\to (-\pi/2,\pi/2)$ is defined by 
\begin{equation}
\label{eq:arctan}
    \arctan x := \int_0^x \frac{1}{1 + y^2} \ud y .
\end{equation}
The transformation of the quadrant and associated angles are depicted in Figure \ref{fig:transform-angles} above.
The following example is relevant for Theorem~\ref{thm:lindley}.

\begin{example}[Orthogonal reflection]\label{ex:orthogonal}
Suppose that $\bmu_1, \bmu_2$
given by~\eqref{eq:mu-def} are non-zero and satisfy
$\bmu_1 = \| \bmu_1 \| e_2$ and $\bmu_2 = \| \bmu_2 \| e_1$. Then $\theta_1=\theta_2=0$, i.e., the effective reflection at both boundaries is orthogonal. In this case the formulas~\eqref{eq:normal-angles-transformed-1} and~\eqref{eq:normal-angles-transformed-2} (or some simple linear geometry) show that $\varphi_1 = \varphi_2 = \varphi_0- \frac{\pi}{2}$, so~$\chi$ defined by~\eqref{eq:chi-def} is $\chi =2 - \frac{\pi}{\varphi_0}$. Then $\chi >0$ if and only if $\varphi_0 > \pi/2$, which, by~\eqref{eq:inangle}, is equivalent to $\rho >0$.
\end{example}

\section{Stabilization of boundary drifts}
\label{sec:stabilization}

The effective reflection vectors
$\bmu_1$ and $\bmu_2$ defined at~\eqref{eq:mu-def}
represent  drifts that the process accumulates when it is near the boundary.
The purpose of this section is to
present and prove Proposition~\ref{prop:stabil},
which formalizes this notion of \emph{stabilization} of boundary drifts.

For $k \in \{1,2\}$ and $z \in \Rq_k$, define the normalized $n$-step drift ($n \in \N$) of the process $Z$ started from $z$
by
\begin{equation}
  \label{eq:n-step-drift-quadrant}
  d^n_k (z)  := \frac{\Exp_z \bigl[ Z_{n} -z \bigr]}{ \Exp_z \bigl[\sum_{\ell
      =0}^{n-1} \Ind{Z_\ell \in \Rq_k} \bigr]}.
\end{equation}
We establish stabilization of $d^n_k (z)$
as $n \to \infty$, in the following sense.

\begin{proposition}[Stabilization]
\label{prop:stabil}
Suppose that~\aref{hyp:irreducible}, \aref{hyp:partial_homogeneity}, \aref{hyp:moments}, and \aref{hyp:zero_drift} hold. Fix $k \in \{1,2\}$, consider $\bmu_k$ as defined in~\eqref{eq:mu-def}, and for $z \in \Rq_k$ let $d^n_k (z)$ be as defined in~\eqref{eq:n-step-drift-quadrant}. Then for any $\veps>0$ there is $N \in \N$ such that 
for every $n >N$,
\begin{equation}
\label{eq:stabilization}
\| d_k^n (z) -\bmu_k \| < \veps, \text{ for all } z \in\Rq_k \text{ with } \| z \| > 2 R n.
\end{equation}
\end{proposition}

The idea behind Proposition~\ref{prop:stabil} is that if $z \in \Rqbx$ (say) with $\| z \|$ large compared to~$n$,
then over $n$ steps, the jumps bounds~\eqref{eq:bounded-jumps} and \eqref{eq:bounded-jumps-12} ensure that $Z$ cannot reach $\Rq_2$ and the vertical component~$Y$ of $Z$ up to time~$n$ evolves as {an irreducible} Markov chain on $\ZP$ (see Proposition~\ref{prop:projection-chains}).
By~\aref{hyp:zero_drift} and~\aref{hyp:partial_homogeneity}, the numerator of~\eqref{eq:n-step-drift-quadrant} changes (by $\mu_1(Y)$) only when $Y\in \I{R}$, and so $d^n_k (z)$ is 
a ratio of expectations of an additive functional of~$Y$ and an occupation time. Proposition~\ref{prop:stabil} will then boil down to an application of the Doeblin ratio limit theorem,
which shows such a ratio is, for large~$n$, close to a corresponding
ratio in terms of the invariant measure~$\pi_1$ of $Y$. We introduce some preparatory notation.

Consider an irreducible, recurrent Markov chain $\xi = (\xi_0,\xi_1,\ldots)$ on $\ZP$.
Fix $R \in \ZP$.
By a result of Derman (quoted in e.g.~\cite[p.~47]{freedman}), there exists an invariant measure $\varpi$
(unique up to a scalar factor) for which $\varpi (y) > 0$ for all $y \in \ZP$.
Define
\begin{equation}
    \label{eq:xi-pi-def}
\pi (y) := \frac{\varpi(y)}{\varpi(\I{R})}, \text{ for all } y \in \I{R}.
\end{equation}
Then $\pi$ is well defined (since $\varpi$ is unique up to a constant factor) and is a probability measure on $\I{R}$ with $\min_{y \in \I{R}} \pi (y) >0$.
 Define \emph{occupation times} associated with $\xi$ by
\begin{equation}
\label{eq:def-local-time}
     \begin{aligned}
 L^\xi_n (y) := \sum_{\ell =0}^{n-1} \Ind { \xi_\ell = y}, 
 \text{ and } 
 L^\xi_n (\I{R} ) := \sum_{\ell =0}^{n-1} \Ind { \xi_\ell \in \I{R}} =\sum_{y \in \I{R}} L^\xi_n (y).
 \end{aligned}
 \end{equation}
The following ratio-ergodic lemma for additive functions of recurrent (but not necessarily positive recurrent) Markov chains is central to our stabilization property.
\begin{lemma}
 \label{lem:local-time-ergodic}
 Let $\xi$ be
  an irreducible, recurrent Markov chain   on $\ZP$, fix $R \in \ZP$,
  and define $\pi$ by~\eqref{eq:xi-pi-def} and $L_n^\xi$ by~\eqref{eq:def-local-time}.
 Take $f : \I{R} \to \R$.
Then, for every $y \in \ZP$,
\[  \lim_{n \to \infty} \frac{\Exp \bigl[ \sum_{\ell =0}^{n-1} f (\xi_\ell ) \Ind { \xi_\ell \in \I{R} }  \bigmid \xi_0 = y \bigr]}
{\Exp \bigl[ L^\xi_n (\I{R}) \bigmid \xi_0 = y \bigr]} = \sum_{x \in \I{R}} f (x) \pi (x) .
\]
\end{lemma}
\begin{proof} Suppose that $\xi_0 = y \in \ZP$.
By Fubini's theorem, linearity, and~\eqref{eq:def-local-time}, we have that
\[
\Exp \biggl[\sum_{\ell=0}^{n-1} f (\xi_\ell) \Ind { \xi_\ell \in \I{R} }\biggmid \xi_0 = y\biggr]
= \sum_{x \in \I{R}} f(x) \Exp \bigl[L^\xi_n (x)\bigmid \xi_0 = y\bigr].
\]
By irreducibility, the Doeblin ratio limit theorem (see~\cite[p.~47]{freedman} or~\cite[p.~50]{chung}) says that, for $\pi$ the probability measure defined by~\eqref{eq:xi-pi-def}, 
 \[
     \lim_{n \to \infty} \frac{\Exp \bigl[ L^\xi_n(x) \bigmid \xi_0 = y \bigr]}{\Exp \bigl[ L^\xi_n(\I{R}) \bigmid \xi_0 = y \bigr] } = \pi (x) , \text{ for every } y \in \ZP. \qedhere
\]
\end{proof}

\begin{remark}
\label{rem:embedded-chain}
The measure $\pi$ as given by~\eqref{eq:xi-pi-def}
can be interpreted in terms of an embedded Markov chain.
An invariant measure $\varpi(y)$ for $\xi$ is furnished by 
the expected number of visits to $y$ between visits to $0$, i.e., for
$\tau := \inf \{ n \in \N : \xi_n = 0\}$,
\[
\varpi (y) = \Exp \biggl[ \sum_{n=0}^{\tau-1} \Ind { \xi_n = y } \biggmid \xi_0 = 0 \biggr].
\]
Define stopping times $\gamma_0 := 0$ and $\gamma_k := \inf \{ n > \gamma_{k-1} : \xi_n \in \I{R} \}$
for $k \in \N$. Then $\txi_k := \xi_{\gamma_k}$ defines a Markov chain on~$\I{R}$ (the \emph{embedded} chain). It is not hard to see that
$\txi$ inherits irreducibility from $\xi$, and so, since $\I{R}$ is finite, $\txi$
is positive recurrent and has a unique stationary distribution. 
Set $\tau' := \inf \{ k \in \N : \txi_k = 0 \}$.
An invariant measure for $\txi$
is given as a function of state $u$ by \[
\Exp  \biggl[  \sum_{\ell= 0}^{\tau'-1} \Ind { \txi_\ell = u } \biggmid \xi'_0 = 0 \biggr],
\]
but, by construction of $\txi$, this is the same as $\varpi (u)$. Thus
the unique stationary distribution for~$\txi$ is proportional to $\varpi$, and hence it must be~$\pi$ as given by~\eqref{eq:xi-pi-def}.    
\end{remark}

\begin{proof}[Proof of Proposition \ref{prop:stabil}]
We write the proof for the case $k=1$; the case $k=2$ is similar.
Let $Z$ be the
process with transition law $p$ as defined in~\eqref{eq:partial-homogeneity}, started from $Z_0 = z \in \Rqbx$.
Fix $n \in \N$, $n >2$.
Recall the definition of $\eta_1$ given in \eqref{eq:def-eta_k}. By~\eqref{eq:bounded-jumps}, we have $\Pr_z ( e_1^\tra Z_\ell > e_1^\tra z - \ell R ) =1$, for all $\ell \in \ZP$.
In particular, if $z \in \Rqbx$ has $\| z \| > (n+2) R$, then $e_1^\tra z > (n+1) R$, by the triangle inequality,
and hence $e_1^\tra Z_\ell > R$ for all $\ell \leq n$. Hence for all $n>1$, 
\begin{equation}\label{eq:away-from-boundary-n-steps}
\Pr_z ( \eta_1> n ) = 1, \text{ for all } z \in \Rqbx \text{ for which } \|z\| > 2n R .    
\end{equation}
Recalling the
definition of the drift $\mu_1$ from~\eqref{eq:mu-function}, we observe
that by~\aref{hyp:zero_drift}, for all $\ell < n$,
\[ 
\Exp_z  [   Z_{\ell+1} -Z_{\ell} \mid Z_\ell = (x,y) ] = \sum_{w \in \Z^2} w p_1 (y ; w) \Ind { y \in \I{R} } = \mu_1 (y) \Ind { y \in \I{R} }.
\]
By
Proposition~\ref{prop:projection-chains}
and~\eqref{eq:away-from-boundary-n-steps}, if $z \in \Rqbx$ and $\|  z \| > 2nR$, then $Y_\ell$, $\ell \leq n$, evolves as a Markov chain on $\ZP$ with transition law $q_1$ as defined in~\eqref{eq:q1-def}. 
Define  the associated occupation time $L_n^Y$ as  in~\eqref{eq:def-local-time}.
The
numerator of~\eqref{eq:n-step-drift-quadrant} can
be expressed as
\begin{align}
  \label{eq:numerator-drift-n-step}
    \sum_{\ell = 0}^{n-1}\Exp_z \big[ Z_{\ell+1} -Z_{\ell} \big] & =\sum_{(x,u) \in \ZP^2}\sum_{\ell = 0}^{n-1}   \Exp_z  [   Z_{\ell+1} -Z_{\ell} \mid Z_\ell = (x,u)  ]  \Pr_z [ Z_\ell = (x,u) ] \nonumber\\
      &=\sum_{u \in \I{R}} \mu_1 (u) \sum_{\ell = 0}^{n-1} \Pr_z [ Y_\ell = u ]   = 
      \sum_{u \in \I{R}} \mu_1(u) \Exp_z[ L^Y_n(u)].
  \end{align}
The denominator of~\eqref{eq:n-step-drift-quadrant} can
be expressed as
\begin{align}
  \label{eq:denominator-drift-n-step}
    \sum_{\ell=0}^{n-1}\Exp_z[\Ind{Z_\ell \in  \Rqbx}]
    &= \sum_{\ell=0}^{n-1}\Exp_z[\Ind{Y_\ell \in \I{R}}]\nonumber\\
    & = \sum_{u \in \I{R}}\sum_{\ell =0}^{n-1}\Exp_z [\Ind{Y_\ell =u}] = \Exp_z [L^Y_n(\I{R})].
  \end{align}
  In~\eqref{eq:numerator-drift-n-step}--\eqref{eq:denominator-drift-n-step},
  $Y$ can be viewed as an irreducible, recurrent Markov process on $\ZP$ with transition law $q_1$
 whose stationary distribution is $\pi_1$ (by Proposition~\ref{prop:projection-chains}).
Therefore, by~\eqref{eq:mu-def}, \eqref{eq:n-step-drift-quadrant},
we can apply Lemma~\ref{lem:local-time-ergodic} to conclude that, for every $\veps>0$ there is $N \in \N$  such that for all
$n \geq N$, and all $z=(x,y) \in \Rqbx$ with $\| z \| > 2Rn$, 
\[
  \left\| d_1(z,n) - \bmu_1 \right\|  = \left\| \sum_{u \in \I{R}}\mu_1(u) \bigg(\frac{ \Exp_z[
    L^Y_n(u)]}{\Exp_z [L^Y_n(\I{R})]}- \pi_1 (u)\bigg) \right\| \leq    \veps.
\]
Since $\veps>0$ was arbitrary, the proof is complete.
\end{proof}

We conclude this section with a proof of  Proposition~\ref{prop:projection-chains}.

\begin{proof}[Proof of Proposition~\ref{prop:projection-chains}.]
We first prove part~\ref{prop:projection-chains-iii}. Observe that, if $(x,i) \in \bI{R} \times \ZP$,
\begin{equation}
    \label{eq:partial-markov}
 \Pr ( Y_{1} = j \mid Z_0 = (x,i) ) = \sum_{x' \in \ZP} \Pr ( Z_{1} = (x', j) \mid Z_0 = (x,i) ) = q_1 (i, j), 
\end{equation}
by~\eqref{eq:q1-def}, and this verifies~\eqref{eq:q1-transition}. A similar argument yields~\eqref{eq:q2-transition}. This proves~\ref{prop:projection-chains-iii}. 

Moreover, 
if we define $q_1^{(1)} (i,j ) := q_1 (i,j)$ and, for $n \in \N$,
\[ q_1^{(n+1)} (i,j) := \sum_{k \in \ZP} q_1 (i, k) q_1^{(n)} (k, j) ,\]
we claim that, for all $n \in \N$ and all $i,j \in \ZP$,
\begin{equation}
    \label{eq:partial-CK}
    \Pr ( Y_n = j, \,  \eta_1 \geq n \mid Z_0 = (x,i) ) \leq q_1^{(n)} (i,j), \text{ for all } x \in \bI{R}. 
\end{equation}
The $n=1$ case of~\eqref{eq:partial-CK} 
is true (with equality) by~\eqref{eq:partial-markov};
an induction of Chapman--Kolmogorov type establishes the general case.
Indeed, from 
the Markov property of~$Z$,
\begin{align*}
& \Pr ( Y_{n+1} = j , \, \eta_1 \geq n+1 \mid Z_0 = (x,i ) ) \\
{} \qquad\qquad {} & = \sum_{x' \in \bb{Z}_+} \sum_{k \in \ZP}
\Exp [ \Pr ( Y_{n+1} = j \mid \cF_n ) \Ind{Z_n = (x',k)} \Ind{ \eta_1 \geq n } \mid Z_0 = (x,i) ] \\
{} \qquad\qquad {} & \leq \sum_{x' \in \bI{R}} \sum_{k \in \ZP} q_1 (k, j) \Pr ( Z_n = (x',k), \,  \eta_1 \geq n \mid Z_0 = (x,i) ) ,
\end{align*}
using~\eqref{eq:q1-transition}. Hence
\[ \Pr ( Y_{n+1} = j , \, \eta_1 \geq n+1 \mid Z_0 = (x,i ) ) 
\leq
\sum_{k \in \ZP} q_1 (k, j) \Pr ( Y_n = k, \,  \eta_1 \geq n \mid Z_0 = (x,i) ) ,
\]
and this provides the inductive step to verify~\eqref{eq:partial-CK}.

Fix $i, j \in \ZP$ and take $x, y \geq R$. From
the hypothesis~\aref{hyp:irreducible} on $p_1$, it follows that there exists $n = n(i,j) \in \N$ for which
\[ \Pr ( Y_n = j, \, \eta_1 \geq n \mid Z_0 = (x,i) ) \geq \Pr ( Z_n = (y,j), \, n < \tau_Z ( \Rqby \cup \Rqc ) \mid Z_0 = (x,i) ) > 0 .
\]
Combining this with~\eqref{eq:partial-CK} verifies that for every $i, j \in \ZP$ there exists $n \in \N$
for which $q_1^{(n)} (i,j) >0$. Hence $q_1$ defines an irreducible Markov chain on $\ZP$; a similar argument
applies to $q_2$. 
This proves part~\ref{prop:projection-chains-i}.
 
 Finally, to prove~\ref{prop:projection-chains-ii}, the existence of an invariant measure unique up to scaling is due to the result of Derman~\cite[p.~47]{freedman} that we already used earlier in this section.
\end{proof}

\section{Bounds on tails of passage times}
\label{sec:proofs}

\subsection{Overview of the proofs}
\label{sec:proof-overview}

To prove the results in~\S\ref{sec:theorem}, we study the  process $Z$ on $\Rq$
using Lyapunov functions $f : \RP^2 \to \RP$,
such that we can apply results for $\RP$-valued adapted processes satisfying suitable sub/super-martingale conditions. For these results to be most effective, one requires good control over the one-dimensional process $f(Z)$, and, in particular, its expected (conditional) increments. An effective~$f$, therefore, is one that homogenizes
the spatial heterogeneity of the process $Z$ (zero drift in $\Rqint$, reflection in $\Rqb$). 
One way to achieve this is to consider an \emph{harmonic} function $h : \RP^2 \to \RP$ 
applied to the transformed process $\tZ = \Ts(Z)$ (see~\S\ref{sec:theorem}),
where $\Ts$ is the linear transformation given in \eqref{eq:linearmap-explicit}.
Then $h\circ \Ts$
will be  harmonic for Brownian motion  with infinitesimal covariance $\Sigma$, and so $h\circ \Ts (Z)$ will be 
approximately a martingale for $Z$ in the interior $\Rqint$. 

If we can also arrange for $h$ to be approximately a martingale for $\tZ$ in the boundaries $\tRqbx$ and $\tRqby$ (see \S\ref{sec:linear_transformation}), we will have a process that is almost a martingale everywhere in the state space (where the error implicit in ``almost'' is reduced with distance from the origin). By considering powers $\alpha >1$ or $\alpha < 1$, we get Lyapunov functions $f = h^\alpha$ such that $f(\tZ)$ is a  sub/super-martingale, respectively, outside of a bounded set.

In the case of a simple boundary (i.e., $R=1$), one can achieve $h(\tZ)$ being almost a martingale 
in $\tRqint$
 by 
choosing $h$ so 
that the reflection vectors (i.e., one-step boundary drifts) 
are tangent to the level curves of~$h$.
For our complex boundary ($R>1$), there are, typically, many reflection vectors ($\mu_k(u)$ for each $u \in \I{R}$)
which means that this approach is impossible. However, if, rather than the one-step drifts, we take drifts
over {$N$}-steps for large enough~{$N$}, the mixing of the process in $\tRqb$ leads, via the stabilization
property described in \S\ref{sec:stabilization}, to   well-defined (approximate) effective drift vectors $\bmu_k$ for which the above strategy of matching level curves can be applied, with appropriate modifications.
Thus we work not with {$f$} applied directly to $\tZ$, but to a \emph{time-changed} version of $\tZ$ in which, in $\tRqb$, time is compressed: this is described in the next subsection. 

\subsection{Time compression at the boundary}
\label{sec:time_compression}

As described in \S\ref{sec:proof-overview}, 
to exploit stabilization (\S\ref{sec:stabilization}) we 
consider a time-changed process $\tZ^N = \big( \tZ^N_n, n \in \ZP \big)$ and 
will deduce properties of the passage times of $\tZ$ from knowledge of the passage times of $\tZ^N$.
The (stochastic) time-change is given by a compression factor $N \in \N$ and $T_N : \ZP \to \ZP$
defined as $T_N(0): = 0$ and for $n \in \N$,
 \begin{equation}\label{eq:time-embedd}
     T_N(n) : = \begin{cases}
    T_N(n-1) + 1 & \text{ if } Z_{T_N(n-1)}\in \Rqint,\\
    T_N(n-1) + N  & \text{ if } Z_{T_N(n-1)}\notin \Rqint.
\end{cases}
 \end{equation}
The process $\tZ^N$ is defined by setting 
 \begin{equation}\label{eq:time-change-process}
     \tZ^N_n : = \tZ_{T_N(n)} = \Ts( Z_{T_N(n)} ), \text{ for all } n \in \ZP.
 \end{equation}
From~\eqref{eq:time-embedd},  $T_N(n) - T_N (n-1) \in \{1,N\}$;
hence $T_N$ is a bounded time-change in the sense 
 \begin{equation}
     \label{eq:time-change-bound}
     n \leq T_N (n) \leq N n, \text{ a.s., for all } n \in \ZP.
 \end{equation}

The next result shows that tails of return times for $Z$ and $\tZ^N$ are comparable, up to constants (depending on $\Sigma$ and the time-change factor $N$).
Recall  that $\tau_Z (r) = 
\inf \{n \in \ZP : \|{Z}_n\| \leq r\}$ from~\eqref{eq:def-passage-time-ball}; analogously, for the   process~$\tZ^N$, define
\begin{equation}
    \label{eq:def-passage-time-ball-time-change}
    \ttau^N_Z (r) := \inf \{n \in \ZP : \|{\tZ}^N_n\| \leq r\}.
\end{equation}
 
\begin{lemma}
\label{lem:time-change}
    There exist constants $c, C$ (depending only on $\Sigma$) with $0 < c \leq C <\infty$  such that
    for every $z \in \Rq$, for all $N \in \N$, $a>0$, and all $n \in \ZP$,
 \begin{equation}
     \label{eq:time-tails-transfer}
       \Pr_z ( \ttau_Z^N (C(a+NR)) \geq n ) \leq \Pr_z ( \tau_Z (a) \geq n ) \leq \Pr_z ( \ttau_Z^N ( c a) \geq n/N)  .
        \end{equation}
\end{lemma}
\begin{proof}
 Let $c: = \inf\{\|\Ts v\| \colon \|v\| = 1\}$ and let $C: = \sup\{\|\Ts v\| \colon \|v\| = 1\}$;
 then, by~\aref{hyp:full_rank}, $0 < c \leq C < \infty$.
 If $\ttau_Z^N(c a) < n$, then $\|\tZ^N_m\| \leq c a$ for some $m < n$, and hence $\|Z_{T_N(m)} \| \leq a$ 
 for some $m < n$. Thus, by~\eqref{eq:time-change-bound},
$\|Z_m\| \leq a$ for some $m < Nn$. Thus
\[
    \text{ if }\ttau_Z^N( c a) < n  \text{ then } \tau_Z(a) < Nn, 
\]
which gives the right-hand inequality in~\eqref{eq:time-tails-transfer}.
Conversely,
suppose that $\tau_Z (a) < n$, so $\| Z_m \| \leq a$
for some $m < n$. By~\eqref{eq:time-embedd},
there exists $k$ such that $m-N \leq T_N (k) \leq m$;
by~\eqref{eq:time-change-bound}, $k \leq m < n$.
By~\eqref{eq:bounded-jumps}
and the triangle inequality, 
since $\| Z_m \| \leq a$ we must have $\| Z_{T_N (k)} \| \leq a + 2 N R$, say,
and thus $\| \tZ^N_k \| \leq C (a + 2 N R) $ for some $k < n$. Hence
\[
    \text{ if }\tau_Z(a) < n  \text{ then } \ttau_Z^N\big(C(a + 2 NR)\big) < n,\]
    which gives the left-hand inequality in~\eqref{eq:time-tails-transfer}.
\end{proof}

\subsection{Multidimensional Taylor's theorem}
\label{sec:taylor}

We collect notation for multivariate calculus and describe the version of Taylor's theorem that we will use to estimate increments of our Lyapunov functions.
Let $f : \R^2 \to \R$ have partial derivatives 
of up to 3rd order all continuous over $\R^2 \setminus \{ \0\}$. Let $D_i$ denote the derivative operator in (Cartesian) direction $e_i$. Write $\nabla f = (D_1 f, D_2 f)$
for the gradient of $f$, and
$H_f$ for the Hessian matrix of $f$, i.e., the matrix whose $i,j$ component is $D_i D_j f$.
For $\nu \in \{1,2\}^m$,
a string of length $|\nu| = m \in \N$, we use multi-index notation $D_\nu$ for 
the $m$th order mixed derivative
$D_{\nu_1} \cdots D_{\nu_m}$.
Taylor's theorem with remainder says 
\begin{align}
    \label{eq:taylor}
    f ( z + y ) - f(z) & = y^\tra \nabla f(z) + \frac{1}{2} y^\tra H_f (z) y + \fR_3^f (z,y) ,
\end{align}
for second-order approximation, 
and, for first-order approximation
\begin{align}
    \label{eq:taylor-first-order}
    f ( z + y ) - f(z) & = y^\tra \nabla f(z) +  \fR_2^f (z,y) ,
\end{align}
both~\eqref{eq:taylor} and~\eqref{eq:taylor-first-order} holding for all $z, y$ with $z \in \R^2 \setminus \{\0\}$
and $z+y \in \R^2 \setminus \{\0\}$;
for $k \in \{2,3\}$, the $k$th-order remainder term $\fR_k^f (z,y)$ satisfies
\begin{equation}
    \label{eq:taylor-remainder}
    \bigl| \fR_k^f (z, y) \bigr| \leq \| y \|^{k} \sup_{ \eta \in [0,1]} \sup_{|\nu| = k} \bigl| D_\nu f (z + \eta y ) \bigr| , \text{ for all } z, z+y \in \R^2 \setminus \{ \0\}.
\end{equation}
We also note the chain rule for the Hessian, which implies that, for $\alpha \in \R$ and for $z \in \{ f > 0 \} := \{ z \in \R^2 : f(z) > 0 \}$,
\begin{equation}
\label{eq:hessian-chain}
H_{f^\alpha} (z) = \alpha f^{\alpha-1} (z) H_f (z) + \alpha (\alpha-1) f^{\alpha-2} (z) ( \nabla f (z) ) (\nabla f (z) )^\tra .
\end{equation}

\subsection{Harmonic functions} 
\label{sec:harmonic}

As described in \S\ref{sec:proof-overview}, our proof
strategy is based around 
construction of Lyapunov functions for the (transformed and time-changed)
process, using an appropriate harmonic function. 
For the wedge or quadrant, there is a well-known parametric family of harmonic functions available for this purpose; these functions were used to study 
exit times from wedges by Burkholder~\cite{burkholder} and reflecting processes in wedges by Varadhan \& Williams~\cite{vw}, and have been used subsequently by many authors (e.g.~\cite{AspIasMen96,iain,vladas,aiSPA,aiTPA,dw,rosenkrantz}).
These functions are most conveniently expressed in polar coordinates, so we first introduce suitable notation for that; our approach is similar to e.g.~\cite[\S6.2]{AspIasMen96}, \cite[\S 3]{MenPet04}, and~\cite[\S3]{iain}.

We write $z = (r, \theta)$ in polar coordinates, with angles measured relative to the 
positive horizontal axis: $r := r(z) := \| z \|$ and $\theta := \theta (z) \in (-\pi,\pi]$
is the angle between $e_1$ and $z$,
i.e., in the notation at~\eqref{eq:angle-def}, $z = r R_\theta (e_1)$.
 The  Cartesian coordinates are $x = r \cos \theta$
and $y = r \sin \theta$, and the Cartesian derivatives of the polar coordinate functions are given by
\begin{equation}
    \label{eq:polar-derivatives}
D_1 r  = \cos \theta ; ~ D_2 r = \sin \theta ; ~ D_1 \theta = - \frac{\sin\theta}{r}; ~ D_2  \theta =  \frac{\cos\theta}{r}. \end{equation}

Recall that $\varphi_0$, as defined in \eqref{eq:inangle}, 
represents the angle at the apex of the wedge~$\cWs = \Ts (\RP^2)$. In our notation for polar coordinates,
\[ \cWs = \{ (r \cos \theta, r \sin \theta ) : r \in \RP, \, 0 \leq \theta \leq \varphi_0 \} .\]
It will be useful to consider a slightly bigger wedge; for this we use the notation
\begin{equation}
\label{eq:bigger-wedge}
    \cWs^\delta := \{ (r \cos \theta, r \sin \theta ) : r \in \RP, \, -\delta \leq \theta \leq \varphi_0+\delta \}. 
\end{equation} 

In equation~\eqref{eq:harmonic-beta} below, we will define the harmonic function $h$, in terms of two angle parameters $\beta_1,\beta_2 \in (-\pi/2,\pi/2)$, and the wedge angle $\varphi_0 \in (0,\pi)$ given by~\eqref{eq:inangle}. The parameters $\beta_1, \beta_2$ will
(in Proposition~\ref{prop:boundary-estimate} and Corollary~\ref{cor:boundary-signs} below) be chosen so that
$\beta_k \approx \varphi_k$, but with the sign of $\varphi_k - \beta_k$ chosen depending on sign of $\chi$, defined at~\eqref{eq:chi-def}, to ensure certain super/submartingale conditions for $h$ or powers $h^\alpha$ ($\alpha >0$) of $h$ applied to the process $\tZ^N$. See~\eqref{eq:pick_beta_k_1}--\eqref{eq:pick_beta_k_4} in the proof of Corollary~\ref{cor:boundary-signs} below for the specific choices of $\beta_k$ that will be made, and Figure~\ref{fig:gradient} below for a picture.

Given $\beta_1, \beta_2, \varphi_0$, set
\begin{equation}
    \label{eq:beta-def}
    \beta:= \frac{\beta_1+\beta_2}{\varphi_0};
\end{equation}
as above, choosing $\beta_k \approx \varphi_k$ will mean $\beta \approx \chi$ when we make specific choices later. Define, in polar coordinates, the function $h : \R^2 \to \R$ with parameters $\beta_1, \beta_2, \varphi_0$~by  
\begin{equation}
  \label{eq:harmonic-beta}
h (z) := h ( r, \theta ) := r^\beta \cos ( \beta \theta - \beta_1 ) .
  \end{equation}
  The function~$h$ is infinitely differentiable on $\R^2 \setminus \{ \0\}$,
  and   its partial derivatives of all orders are continuous on $\R^2 \setminus \{ \0\}$.
By elementary calculus using the derivatives~\eqref{eq:polar-derivatives},
\begin{align}
    \label{eq:h-derivatives}
    D_1 h(z) = \beta r^{\beta -1} \cos ( (\beta-1) \theta - \beta_1) , ~ D_2 h(z) = - \beta r^{\beta -1} \sin ( (\beta-1) \theta - \beta_1),
\end{align}
and
\begin{align}
    \label{eq:h-2nd-derivatives}
    D_1^2 h(z) = \beta (\beta-1) r^{\beta -2} \cos ( (\beta-2) \theta - \beta_1 ) = - D_2^2 h(z) ,
\end{align}
verifying that $D_1^2 h + D_2^2 h \equiv 0$, i.e., $h$ is harmonic. We also observe
that~\eqref{eq:h-derivatives} implies that $\| \nabla h(z) \|^2 = ( D_1 h(z) )^2 + ( D_2 h(z) )^2 = \beta^2 r^{2\beta-2}$,
from which we see that 
\begin{equation}
    \label{eq:norm-grad}
 \| \nabla h(z) \| = \beta \| z\|^{\beta-1}, \text{ for all } z \in \R^2 \setminus \{\0\}.   
\end{equation}
See Figure~\ref{fig:gradient} for an illustration of the relationship between the gradient $\nabla h$ and the angles $\beta_1, \beta_2$ that define $h$ through~\eqref{eq:beta-def}--\eqref{eq:harmonic-beta}. 

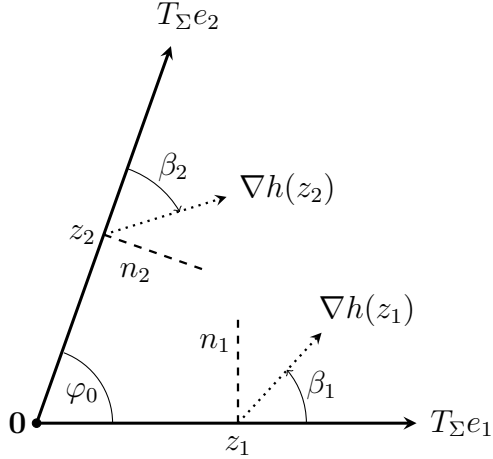
\begin{figure}[h!]
\begin{center}
\scalebox{1.0}{
\begin{tikzpicture}[domain=0:10, scale = 1.0]
\def \t {.15}
\filldraw (0,0) circle (1.6pt);
\draw[black,->,>=stealth,line width =.4mm] (0,0) -- (1.76776695,5);
\draw[black,->,>=stealth,line width =.4mm] (0,0) -- (5,0);
\draw[black,-,>=stealth,dashed,line width =.3mm] (0.88388,2.5) -- (0.88388+1.319932658,2.5-0.46666667);
\draw[black,-,>=stealth,dashed,line width =.3mm] (2.5 + \t,0) -- (2.5+\t,1.4);
\node at (-0.25,0) {$\0$};
\node at (5.6,0) {$T_\Sigma e_1$};
\node at (2,5.4) {$T_\Sigma e_2$};
    \draw (1,0) arc (0:69:1);
    \node at (0.6,0.4) {$\varphi_0$};
    \draw[black,->,>=stealth,dotted,line width =.3mm] (2.5 + \t,0) -- (3.6 + \t,1.2);
     \draw[black,->,>=stealth,dotted,line width =.3mm] (0.88388,2.5) --(2.5,3);
  \draw[<-] (1.88388,2.8) arc (30:70:1.3);
      \node at (1.3,2.0) {$n_2$};
            \node at (0.6,2.5) {$z_2$};
   \node at (3.6+\t,0.5) {$\beta_1$};
   \node at (1.8,3.4) {$\beta_2$};
   \draw[->] (3.4+\t,0) arc (0:40:1.1);
   \node at (4.2+\t,1.5) {$\nabla h(z_1)$};
 \node at (3.3,3.1) {$\nabla h(z_2)$};
  \node at (2.2 + \t,1.06) {$n_1$};
    \node at (2.5 + \t,-0.3) {$z_1$};
\end{tikzpicture}}
\end{center}
\caption{\label{fig:gradient}
The relationship between parameters $\beta_1, \beta_2$
and the gradient $\nabla h$
at $z_1 = (1,0)$ and $z_2 = (\cos \varphi_0, \sin \varphi_0)$.
The unit inwards-pointing normal vectors at the boundaries are $n_1 = (0,1)$ and $n_2 = (\sin \varphi_0, - \cos \varphi_0)$. The derivative formulas~\eqref{eq:h-derivatives} show that when $\theta =0$, $\nabla h$ is in the direction $(\cos \beta_1, \sin \beta_1)$, while when $\theta = \varphi_0$, 
$\nabla h$ is in the direction $(\cos (\varphi_0 - \beta_2), \sin (\varphi_0 - \beta_2))$.}
\end{figure} 
    
The following simple fact
shows that for $z \in \cWs$, it holds that $h (z) >0$ (unless $z=\0$) and $h(z)$ grows polynomially in $\| z \|$.

\begin{lemma}[Polynomial growth]
\label{lem:growth-factor}
Fix $\varphi_0 \in (0,\pi)$ and $\beta_1,\beta_2 \in (-\pi/2,\pi/2)$.
With $\beta$ given at~\eqref{eq:beta-def} and $h$ given by~\eqref{eq:harmonic-beta}, 
there exist $\delta >0$ and $\veps_0 >0$ for which 
\begin{equation}
\label{eq:growth-factor}
    \veps_0 \| z \|^\beta \leq h(z) \leq \|z\|^\beta , \text{ for all } z \in \cWs^\delta,
\end{equation}    
and for every $m \in \N$ there exists $A := A(m,\beta) < \infty$ such that, for every $\nu \in \{1,2\}^m$, 
\begin{equation}
\label{eq:growth-bound-derivatives}
  |  D_\nu h(z) | \leq A \|z\|^{\beta -m} , \text{ for all } z \in \cWs^\delta \setminus \{ \0\}.
\end{equation}    
\end{lemma}
  \begin{proof}
  Let $\beta_1,\beta_2 \in (-\pi/2, \pi/2)$ and $\varphi_0 \in (0,\pi)$ be fixed,
and define $\beta$ as at~\eqref{eq:beta-def}. 
The upper bound in~\eqref{eq:growth-factor} 
follows from \eqref{eq:harmonic-beta} and the fact that $|\cos \theta |\leq 1 $ for all $\theta \in \R$.
Choose $\delta >0$ with $2 \delta | \beta | <  \frac{\pi}{2} - \max ( |\beta_1|, |\beta_2| )
$. If $\beta \geq 0$, then 
\[ - \frac{\pi}{2} + \delta | \beta | < - \beta_1-\delta \beta \leq \beta \theta - \beta_1 \leq \beta_2+\delta \beta < \frac{\pi}{2} - \delta | \beta |  , \text{ for all } \theta \in [-\delta,\varphi_0+\delta]; \]
similarly, if $\beta <0$, then,
$| \beta \theta - \beta_1 | < \frac{\pi}{2} - \delta | \beta |$. Hence
\[ 
\inf_{ \theta \in [-\delta,\varphi_0+\delta]} \cos (\beta \theta - \beta_1 ) \geq \veps_0, \text{ where } \veps_0 :=  \cos \bigl( \tfrac{\pi}{2} - \delta |\beta| \bigr)  >0 .
\]
This completes the proof of~\eqref{eq:growth-factor}. From repeated differentiation and application of~\eqref{eq:polar-derivatives}, one sees that, for $| \nu | = m$, 
$D_\nu h(z) = r^{\beta-m} A_\nu (\beta, \beta_1, z)$ for a function $A_\nu (\beta, \beta_1, z)$ which satisfies $\sup_{\beta_1} \sup_{z \in \R^2} \| A_\nu (\beta, \beta_1, z) \| < \infty$; this yields~\eqref{eq:growth-bound-derivatives}.
\end{proof}

For certain applications of the optional stopping theorem, it is technically
convenient to introduce a   version $h_b$ of the function $h$, truncated at a fixed level $b \in \RP$, defined by
\begin{equation}
    \label{eq:h-truncated}
    h_b (z) := (2b)^\beta  \wedge h (z).
\end{equation}
Since $ h(z)  \leq \| z \|^\beta$,
$z \in \cWs^\delta$, we have that $h_b (z) = h(z)$ when $z \in \tRq$ with $\| z \| \leq 2b$.

\subsection{Estimates for Lyapunov function increments: interior}
\label{sec:lyapunov-interior}

The family of harmonic functions~$h$
described in \S\ref{sec:harmonic}
provides candidate Lyapunov functions $h^\alpha$, $\alpha >0$;
we estimate expectations of the functional increments 
$h^\alpha ( \tZ_{n+1}^N ) - h^\alpha ( \tZ_{n}^N )$
when $\tZ_n^N \in \tRqint = \Ts (\Rqint)$ (this section) and $\tZ_n^N \in \tRqb  = \Ts (\Rqbx \cup \Rqby) $ (\S\ref{sec:lyapunov-boundary}). (See~\S\ref{sec:linear_transformation} for definitions of $\Ts$, $\tRqint$, etc.)
As indicated in \S\ref{sec:proof-overview},
$h ( \tZ_{n}^N )$ is ``almost''
a martingale, because $h$ is harmonic and the increments of $\tZ^N_n$ are standardized.
The $\alpha =1$ case of the main result of this section, Proposition~\ref{prop:interior-estimate}, gives a concrete
estimate to this effect, but does not yield control over the \emph{sign} of the deviation
from the martingale property, i.e., the $o$ term in~\eqref{eq:conditional-increment-h-interior} below.
However, taking $\alpha <1$ or $\alpha >1$ furnishes a super- or submartingale, respectively, each with a quantitative estimate; this is the main content of Proposition~\ref{prop:interior-estimate} below.
See Lemma~6(a) of~\cite{AspIasMen96} and~Lemma~5.2 of~\cite{iain} for related results.

\begin{proposition}[Interior estimate]
\label{prop:interior-estimate}
Suppose that hypotheses~\aref{hyp:irreducible}, \aref{hyp:partial_homogeneity}, \aref{hyp:moments},
\aref{hyp:zero_drift}, and~\aref{hyp:full_rank} hold.
Let $N \in \N$, $\alpha >0$, $\beta_1,\beta_2 \in(-\pi/2,\pi/2)$, and define $\beta$ as at~\eqref{eq:beta-def}. Suppose that, for $\nuint >2$ the constant in~\aref{hyp:moments}, $2- \nuint < \alpha \beta < \nuint$. 
Then, for all $z \in \tRqint$ and all $n \in \ZP$,
\begin{equation}
  \label{eq:conditional-increment-h-interior}
\Exp\bigl[ h^\alpha(\tZ^N_{n+1}) - h^\alpha(\tZ^N_{n}) \bigmid \tZ^N_{n} = z \bigr] 
 =  
\frac{ \alpha(\alpha-1) \beta^2}{2} \|z\|^{2\beta-2}  h^{\alpha-2}(z) 
+  o(\|z\|^{\alpha\beta -2}),
\end{equation}
 as $\| z \| \to \infty$.
In particular, the following hold.
\begin{itemize}
\item 
If $\alpha >1$, there exist constants $c >0$
and $r \in \RP$ such that 
\begin{equation}
\label{eq:interior-sign-plus}
\Exp\bigl[ h^\alpha(\tZ^N_{n+1}) - h^\alpha(\tZ^N_{n}) \bigmid \tZ^N_{n} = z \bigr] 
\geq c  \|z\|^{\alpha\beta -2}, \text{ for all } z \in \tRqint \setminus B_r.
\end{equation}
\item 
If $\alpha \in (0,1)$,  there exist constants 
$c    >0$
and $r \in \RP$ such that 
\begin{equation}
\label{eq:interior-sign-minus}
\Exp\bigl[ h^\alpha(\tZ^N_{n+1}) - h^\alpha(\tZ^N_{n}) \bigmid \tZ^N_{n} = z \bigr] 
\leq - c  \|z\|^{\alpha\beta -2}, \text{ for all } z \in \tRqint \setminus B_r.
\end{equation}
\end{itemize}
\end{proposition}

Let $\Delta : =\tZ_{1} - \tZ_{0}$. 
From~\eqref{eq:time-embedd}, we have that
$\tZ^N_{n+1} - \tZ^N_n = \tZ_{T_N(n)+1} - \tZ_{T_N(n)}$
whenever $\tZ^N_n \in \tRqint$.
Hence, for all $z \in \tRqint$ and all $n \in \ZP$,
\begin{equation}  
\label{eq:conditional-increment-f}
\Exp[ h^\alpha (\tZ^N_{n+1}) - h^\alpha (\tZ^N_{n}) \mid \tZ^N_{n} = z] =
    \Exp[h^\alpha (z + \Delta) - h^\alpha(z) \mid \tZ_0 = z].
\end{equation}
For $\delta \in (0,1)$ and $z \in \tRq$,
define the event
\begin{equation}
    \label{eq:E-delta}
 E_\delta (z) := \bigl\{ \| \Delta \| \leq (1 + \| z \|)^{1-\delta} \bigr\} .
 \end{equation}
On the event $E_\delta(z)$,
 control of the remainder terms of the form~\eqref{eq:taylor-remainder} 
 will enable us to 
estimate the expected increment in~\eqref{eq:conditional-increment-f} using 
the multivariate Taylor expansion
from \S\ref{sec:taylor}.
On the complementary event, $E_\delta^\rc (z)$, 
to bound the expected increment in~\eqref{eq:conditional-increment-f}
we will instead rely
on the moments assumption~\aref{hyp:moments} which shows that big jumps are unlikely. 
The following elementary technical lemma will be used in the latter case.

\begin{lemma}
    \label{lem:big-jump}
    Suppose that for some~$\nu >0$, $B < \infty$, and $z \in \tRq$, 
    the increment $\Delta$ satisfies
    $\Exp [ \| \Delta \|^\nu \mid \tZ_0 = z ] \leq B$.
    Then, for every~$\delta \in (0,1)$ and every~$\gamma \in [0, \nu]$,
    \[  \Exp \bigl[ \| \Delta \|^\gamma \Indno {E_\delta^\rc (z)} \bigmid \tZ_0 = z \bigr] \leq B  (1 + \| z \|)^{-(1-\delta)(\nu-\gamma)}, \text{ for all } z \in \tRq.
    \]
    In particular, if $\nu >2$ and $\delta \in (0, \frac{\nu-2}{\nu-1})$, then for $z \in \tRq$,
     as $\| z \| \to \infty$,
    \begin{equation}
        \label{eq:first-moment-big-jump}
     \Exp [ \| \Delta \| \Indno{E^\rc_\delta(z)}  \mid \tZ_0 = z ]  = o ( \| z \|^{-1} ),
     \text{ and } \Exp [ \| \Delta \|^2 \Indno{E^\rc_\delta(z)}  \mid \tZ_0 = z ]  = o (1 ) .\end{equation}
\end{lemma}
\begin{proof}
    Write $\| \Delta \|^\gamma = \| \Delta \|^\nu \cdot \| \Delta \|^{\gamma - \nu}$,
    with $\gamma - \nu \leq 0$, to see that
    \[  \| \Delta \|^\gamma \Indno {E_\delta^\rc (z)} 
    \leq \| \Delta \|^\nu (1 + \| z \|)^{-(1-\delta)(\nu-\gamma)} ,
    \]
    and then take (conditional) expectations. The last part of the lemma is obtained {for} $\gamma =1$ by noting that when $\delta \in (0, \frac{\nu-2}{\nu-1})$ it holds that $(1-\delta)(\nu-\gamma)>1$, and for $\gamma =2$ it suffices to note that $(1-\delta)(\nu-\gamma)>0$ since $\delta <1$ and $\nu >2$.
\end{proof}

For a real-valued $2$-dimensional matrix $M$,  denote by
$\| M \|_{\text{op}} := \sup_{z \in \R^2 \setminus \{ \0 \}} \| M z \|/\|z\|$, the matrix (operator) norm induced by the Euclidean norm on $\R^2$.

\begin{proof}[Proof of Proposition~\ref{prop:interior-estimate}.]
As promised, we partition the expected increment in~\eqref{eq:conditional-increment-f}:
\begin{equation}
\begin{aligned}
\label{eq:conditional-increment-f-partition}
\Exp[ h^\alpha (\tZ^N_{n+1}) - h^\alpha (\tZ^N_{n}) \mid \tZ^N_{n} = z] 
& =
    \Exp \bigl[ \bigl( h^\alpha (z + \Delta) - h^\alpha(z) \bigr) \Indno{E_\delta (z)} \bigmid \tZ_0 = z \bigr] \\
    & {} \quad {} + \Exp \bigl[ \bigl( h^\alpha (z + \Delta) - h^\alpha(z) \bigr) \Indno{E_\delta^\rc (z)} \bigmid \tZ_0 = z \bigr],
\end{aligned}
\end{equation}
where we fix $\delta \in (0,\frac{\nuint-2}{\nuint-1})$ and  $E_\delta (z)$ is defined at~\eqref{eq:E-delta}.
To estimate the first term on the right-hand side of~\eqref{eq:conditional-increment-f-partition}, we 
take $f = h^\alpha$ in~\eqref{eq:taylor}
and use the Hessian chain rule at~\eqref{eq:hessian-chain} to obtain, for all $z \in \tRqint$
and all $y \in \R^2$ with $\| y \| < \| z \|$,
\begin{align}
\label{eq:taylor-for-h-alpha}
h^\alpha (z+ y) - h^\alpha (z) 
& = \alpha h^{\alpha -1}(z) y^\tra \nabla h(z)
+   \frac{\alpha }{2} h^{\alpha -1} (z) y^\tra H_h (z) y  \nonumber\\
& {} \quad {} +\frac{\alpha (\alpha-1)}{2}  h^{\alpha-2}(z) (y^\tra \nabla h(z) )^2 
 + \fR_3^{h^\alpha} (z, y).
\end{align}
By~\eqref{eq:taylor-remainder} together with the upper
bounds on $h$ and its derivatives in~\eqref{eq:growth-factor} and~\eqref{eq:growth-bound-derivatives},
the error term $\fR_3^{h^\alpha}$ in~\eqref{eq:taylor-for-h-alpha} satisfies the bound
\begin{equation}
    \label{eq:taylor-error}
\bigl| \fR_3^{h^\alpha} (z,y) \bigr| \leq C_\delta \| z \|^{\alpha \beta -3} \cdot \| y \|^3 
, 
\end{equation} 
for some constant $C_\delta < \infty$ and all $z, z+y \in \tRqint$.

We will apply~\eqref{eq:taylor-for-h-alpha} with 
$y = \Delta$, on the event $E_\delta (z)$, and then take (conditional) expectations. 
First note that, by~\eqref{eq:E-delta} and~\eqref{eq:taylor-error}, 
\begin{align}
\label{eq:error-term-estimate}
 \Exp \Bigl[\bigl| \fR_3^{h^\alpha} (z,\Delta) \bigr| \Indno{E_\delta(z)}
\Bigmid \tZ_0 = z \Bigr] 
& \leq C_\delta (1+\| z \|)^{\alpha \beta -2 -\delta}  \Exp [ \| \Delta \|^2 \Indno{E_\delta(z)}
\mid \tZ_0 = z ]  \nonumber\\
& = o ( \| z \|^{\alpha \beta -2} ),
  \text{ for } z \in \tRqint \text{ with } \|z\| \to \infty,
\end{align}
by the fact that~\aref{hyp:moments} holds with $\nuint > 2$.
Next, observe that, by linearity, 
\begin{align*}
  &\Exp\big[ \bigl( \Delta^\tra \nabla h (z) \bigr) \Indno{E_\delta(z)} \bigmid \tZ_0 = z  \big] 
    =( \nabla h(z))^\tra \Exp [ \Delta   \mid \tZ_0 = z ] - ( \nabla h(z))^\tra   \Exp [ \Delta \Indno{E^\rc_\delta(z)}  \mid \tZ_0 = z ].
\end{align*}
Hence, by~\aref{hyp:moments} and the $\gamma =1$, $\nu = \nuint >2$ case of~\eqref{eq:first-moment-big-jump},
together with~\eqref{eq:norm-grad} and~\eqref{eq:transformed-drift}, 
which follows from the zero-drift hypothesis~\aref{hyp:zero_drift}, we obtain 
\begin{equation}
  \label{eq:zero-drift-h}
  \bigl| \Exp\big[ \bigl( \Delta^\tra \nabla h (z) \bigr) \Indno{E_\delta(z)} \bigmid \tZ_0 = z  \big] \bigr| = o (\| z\|^{\beta-2}),
  \text{ for } z \in \tRqint \text{ with } \|z\| \to \infty.
\end{equation}
Moreover, by linearity 
of the trace 
and~\eqref{eq:transformed-covariance}, we obtain
\begin{align*}
\Exp \bigl[ \bigl( \Delta^\tra H_h (z) \Delta \bigr) \Indno{E_\delta(z)} \bigmid \tZ_0 = z \bigr]
&= \Exp \bigl[   \trace \bigl( \Delta \Delta^\tra \Indno{E_\delta(z)} H_h (z) \bigr) \bigmid \tZ_0 = z \bigr] \nonumber\\
& = \trace \bigl( \Exp \bigl[  \Delta \Delta^\tra  \Indno{E_\delta(z)} \bigmid \tZ_0 = z  \bigr] H_h (z) \bigr) \nonumber\\
&  = \trace \bigl( \Exp \bigl[  \Delta \Delta^\tra -  \Delta \Delta^\tra  \Indno{E^c_\delta(z)} \bigmid \tZ_0 = z  \bigr] H_h (z) \bigr) \nonumber\\
& = \trace H_h (z) - \trace \bigl( \Exp \bigl[  \Delta \Delta^\tra  \Indno{E^\rc_\delta(z)} \bigmid \tZ_0 = z  \bigr] H_h (z) \bigr).
\end{align*}
Here, since $h$ is harmonic (see \S\ref{sec:harmonic}), $\trace H_h (z) =0$, while,
by~\aref{hyp:moments}, the $\gamma = 2$, $\nu = \nuint >2$ case~\eqref{eq:first-moment-big-jump},
and~\eqref{eq:growth-bound-derivatives},
\[ 
\bigl| \trace \bigl( \Exp \bigl[  \Delta \Delta^\tra  \Indno{E^\rc_\delta(z)} \bigmid \tZ_0 = z  \bigr] H_h (z) \bigr) \bigr| \leq \| H_h (z) \|_{\text{op}} 
\Exp \bigl[ \| \Delta \|^2  \Indno{E^\rc_\delta(z)} \bigmid \tZ_0 = z  \bigr] = o ( \| z \|^{\beta -2} ),
\]
as $\| z \| \to \infty$. Hence we conclude that
\begin{align}
  \label{eq:trace-hessian}
\Exp \bigl[ \bigl( \Delta^\tra H_h (z) \Delta \bigr) \Indno{E_\delta(z)} \bigmid \tZ_0 = z \bigr]
& =o ( \| z \|^{\beta -2} ),  \text{ for } z \in \tRqint \text{ with } \|z\| \to \infty.
\end{align}
Similarly,
\begin{align*}
\Exp \bigl[ \bigl( \Delta^\tra \nabla h(z) \bigr)^2  \Indno{E_\delta(z)} \bigmid \tZ_0 = z \bigr]
& = \Exp \bigl[ \trace  \bigl( \Delta^\tra ( \nabla h(z) ) ( \nabla h(z) )^\tra \Delta  \Indno{E_\delta(z)}\bigr)  \bigmid \tZ_0 = z \bigr] \nonumber\\
& =   \trace \bigl( \Exp \bigl[  \Delta \Delta^\tra  (1-\Indno{E^\rc_\delta(z)}) \bigmid \tZ_0 = z  \bigr] ( \nabla h(z) ) ( \nabla h(z) )^\tra \bigr) , \end{align*}
while $ \trace \left( ( \nabla h(z) ) ( \nabla h(z) )^\tra\right)  = \| \nabla h(z) \|^2$, and,
by Lemma~\ref{lem:big-jump} and~\eqref{eq:h-derivatives},
\begin{align*}
\bigl| \trace \bigl( \Exp \bigl[  \Delta \Delta^\tra  \Indno{E^\rc_\delta(z)} \bigmid \tZ_0 = z  \bigr] ( \nabla h(z) ) ( \nabla h(z) )^\tra \bigr) \bigr| & \leq \| \nabla h(z) \|^2 \Exp \bigl[ \|  \Delta \|^2  \Indno{E^\rc_\delta(z)} \bigmid \tZ_0 = z  \bigr] \\
& = o ( \| z \|^{2\beta -2} ),  \text{ as } \|z\| \to \infty.
\end{align*}
Hence, using~\eqref{eq:transformed-covariance} and~\eqref{eq:norm-grad}, we conclude that
\begin{align}
  \label{eq:trace-grad}
\Exp \bigl[ \bigl( \Delta^\tra \nabla h(z) \bigr)^2  \Indno{E_\delta(z)} \bigmid \tZ_0 = z \bigr]
& = \bigl( \beta^2 +o(1) \bigr) \| z \|^{2\beta-2} ,    \text{ for } z \in \tRqint \text{ with } \|z\| \to \infty.
\end{align}
Using~\eqref{eq:error-term-estimate}, \eqref{eq:zero-drift-h}, \eqref{eq:trace-hessian},
and~\eqref{eq:trace-grad} in~\eqref{eq:taylor-for-h-alpha},
we obtain the interior Taylor expansion
\begin{align}
\label{eq:interior-terms}
 & \Exp \bigl[ \bigl( h^\alpha (z + \Delta) - h^\alpha(z) \bigr) \Indno{E_\delta (z)} \bigmid \tZ_0 = z \bigr] \nonumber\\
& {} \quad {} =  \frac{ \alpha(\alpha-1) \beta^2}{2} \|z\|^{2\beta-2}  h^{\alpha-2}(z) 
+  o(\|z\|^{\alpha\beta -2}),   \text{ for } z \in \tRqint \text{ with } \|z\| \to \infty.
\end{align}   

On the other hand, we consider the case when $E_\delta^\rc (z)$ occurs.
Write $x^+ := x \Ind{ x >0 }$ and $x^-:= - x\Ind { x < 0}$. 
We  
use the triangle inequality, the fact that $| h(z)| \leq \| z \|^\beta$ from~\eqref{eq:growth-factor}, and~\eqref{eq:E-delta}, to obtain the elementary bound
\begin{align}
\label{eq:big-jump-interior}
   \bigl| h^\alpha (z + \Delta) - h^\alpha(z) \bigr| \Indno{E^\rc_\delta (z)}
   & \leq \bigl( \| z + \Delta \|^{(\alpha \beta)^+}  +  \| z \|^{(\alpha \beta)^+} \bigr) \Indno{E^\rc_\delta (z)} \nonumber\\
   & \leq C_1  \| z \|^{(\alpha \beta)^+} \Indno{E^\rc_\delta (z)} + C_2 \| \Delta \|^{(\alpha \beta)^+} \Indno{E^\rc_\delta (z)} \nonumber\\
   & \leq C_3 \| \Delta \|^{(\alpha \beta)^+/(1-\delta)} \Indno{E^\rc_\delta (z)}, \text{ for all } z \in \tRq, 
\end{align}
where $C_1, C_2, C_3 < \infty$ depend only on $\alpha\beta$. 
If $(\alpha \beta)^+ < \nuint$ and
$\delta < 1 -\frac{(\alpha \beta)^+}{\nuint}$, we can apply the $\nu = \nuint >2$
case of Lemma~\ref{lem:big-jump}, with $\gamma = (\alpha \beta)^+/(1-\delta)$ to deduce from~\eqref{eq:big-jump-interior} that
\begin{align}
\label{eq:big-jump-bound-minus-2}
\Exp \bigl[ \bigl| h^\alpha (z + \Delta) - h^\alpha(z) \bigr| \Indno{E_\delta^\rc (z)} \bigmid \tZ_0 = z \bigr]
& \leq C \Exp \bigl[  \| \Delta \| ^{(\alpha \beta)^+/(1-\delta)} \Indno{E_\delta^\rc (z)} \bigmid \tZ_0 = z \bigr]\nonumber\\
& = O( \|z \|^{(\alpha\beta)^+-(1-\delta) \nuint} ).
\end{align}
 Provided that
$\delta < \frac{\nuint - 2 - (\alpha\beta)^-}{\nuint}$,
the bound in~\eqref{eq:big-jump-bound-minus-2}
is  $o(\|z\|^{\alpha\beta -2})$, and so
\begin{equation}
    \label{eq:h-big-jump}
    \bigl| \Exp \bigl[ \bigl( h^\alpha (z + \Delta) - h^\alpha(z) \bigr) \Indno{E_\delta^\rc (z)} \bigmid \tZ_0 = z \bigr] \bigr| = o(\|z\|^{\alpha\beta -2}),
    \text{ for } z \in \tRqint \text{ with } \|z\| \to \infty.
\end{equation}
For $\alpha \beta \geq 0$, the additional constraint on $\delta \in (0 , \frac{\nuint-2}{\nuint-1})$ is $\delta < \frac{\nuint - 2}{\nuint}$,
for which there is an admissible $\delta \in (0,1)$ whenever $\nuint >2$,
while if $\alpha \beta <0$, the constraint is $\delta < \frac{\nuint - 2 +\alpha \beta}{\nuint}$,
which requires $\alpha \beta > 2 -\nuint$. 
Thus whenever $2-\nuint < \alpha \beta < \nuint$, we can choose a suitable $\delta \in (0,1)$
for which both~\eqref{eq:interior-terms} and~\eqref{eq:h-big-jump} hold, and from here~\eqref{eq:conditional-increment-h-interior} follows.
Then~\eqref{eq:conditional-increment-h-interior} together with the lower
bound in Lemma~\ref{lem:growth-factor}
yields~\eqref{eq:interior-sign-plus} and~\eqref{eq:interior-sign-minus}.
\end{proof}

The following variant of Proposition~\ref{prop:interior-estimate}
applies to the truncated Lyapunov function defined using $h_b$ from~\eqref{eq:h-truncated}.

\begin{proposition}
\label{prop:interior-estimate-truncated}
Suppose that hypotheses~\aref{hyp:irreducible}, \aref{hyp:partial_homogeneity}, \aref{hyp:moments},
\aref{hyp:zero_drift}, and~\aref{hyp:full_rank} hold.
Let  $N \in \N$, $\alpha >0$, $\beta_1,\beta_2 \in(-\pi/2,\pi/2)$, and define $\beta$ as at~\eqref{eq:beta-def}. Suppose that, for $\nuint >2$ the constant in~\aref{hyp:moments}, $2 - \nuint  < \alpha \beta < \nuint$. 
Then  the following hold.
\begin{itemize}
\item 
If $\alpha >1$, there exist constants $c >0$
and $r \in \RP$ such that, for all $b \geq r$,
\begin{equation}
\label{eq:interior-sign-plus-truncated}
\Exp\bigl[ h_b^\alpha(\tZ^N_{n+1}) - h_b^\alpha(\tZ^N_{n}) \bigmid \tZ^N_{n} = z \bigr] 
\geq c  \|z\|^{\alpha\beta -2}, \text{ for all } z \in  \tRqint \cap ( B_b \setminus B_r ).
\end{equation}
\item 
If $\alpha \in (0,1)$,  there exist constants 
$c    >0$
and $r \in \RP$ such that, for all $b \geq r$,
\begin{equation}
\label{eq:interior-sign-minus-truncated}
\Exp\bigl[ h_b^\alpha(\tZ^N_{n+1}) - h_b^\alpha(\tZ^N_{n}) \bigmid \tZ^N_{n} = z \bigr] 
\leq - c  \|z\|^{\alpha\beta -2}, \text{ for all } z \in  \tRqint \cap ( B_b \setminus B_r ).
\end{equation}
\end{itemize}
\end{proposition}
\begin{proof}
Take  $\delta \in (0,\frac{\nuint-2}{\nuint-1})$, and recall the definition of $E_\delta (z)$ from~\eqref{eq:E-delta}. Then, 
\begin{equation}\label{eq:increment_spli_adjust}
\begin{aligned}
\Exp[ h_b^\alpha (\tZ^N_{n+1}) - h_b^\alpha (\tZ^N_{n}) \mid \tZ^N_{n} = z] 
& =
\Exp \bigl[ \bigl( h_b^\alpha (z + \Delta) - h_b^\alpha(z) \bigr) \Indno{E_\delta (z)} \bigmid \tZ_0 = z \bigr]\\
& {} \quad {} + \Exp \bigl[ \bigl( h_b^\alpha (z + \Delta) - h_b^\alpha(z) \bigr) \Indno{E_\delta^\rc (z)} \bigmid \tZ_0 = z \bigr],
\end{aligned}
\end{equation}
similarly to~\eqref{eq:conditional-increment-f-partition}.
There exists $r_0$ (depending only on $\delta$)
such that whenever $z \in \tRqint \setminus B_{{r_0}}$,
on the event $E_\delta (z)$ we have $\| z + \Delta \| \leq 2 \| z \|$.
Hence, if $z \in \tRqint \cap ( B_b \setminus B_r )$ for $b \geq r \geq r_0$,
\begin{equation}\label{eq:no-change_small_increment}
 \Exp \bigl[ \bigl( h_b^\alpha (z + \Delta) - h_b^\alpha(z) \bigr) \Indno{E_\delta (z)} \bigmid \tZ_0 = z \bigr]
= \Exp \bigl[ \bigl( h^\alpha (z + \Delta) - h^\alpha(z) \bigr) \Indno{E_\delta (z)} \bigmid \tZ_0 = z \bigr],
\end{equation}
and then~\eqref{eq:interior-terms} applies;
the $o( \, \cdot \,)$ term on the right-hand side of~\eqref{eq:interior-terms} means that
we can take $r$ large enough so that the sign of the increment is controlled for all $z \in \tRqint \cap ( B_b \setminus B_r )$, independently of $b$.
On the other hand, when $E_\delta^\rc (z)$ occurs, similarly to~\eqref{eq:big-jump-interior},
\begin{equation}\label{eq:adjust_upper_bound_large_increment}
  \bigl| h_b^\alpha (z + \Delta) - h_b^\alpha(z) \bigr| \Indno{E^\rc_\delta (z)}
   \leq C    \| \Delta \|^{(\alpha \beta)^+/(1-\delta)} \Indno{E^\rc_\delta (z)},
\end{equation} 
where constant $C$ does not depend on $b$, and  the argument leading to~\eqref{eq:h-big-jump}
applies verbatim. The proof is concluded in the same way as the proof
of Proposition~\ref{prop:interior-estimate}. 
\end{proof}

\subsection{Estimates for Lyapunov function  increments: boundary}
\label{sec:lyapunov-boundary}

This section presents Lyapunov function estimates at the boundaries, to accompany those in the interior from the previous section;
here the time compression factor $N$ (see~S\ref{sec:time_compression}) 
is crucial, at least when $R > 1$.
In the case of a simple boundary ($R=1$),
Proposition~\ref{prop:boundary-estimate} is essentially
Lemma~6(b) in~\cite{AspIasMen96}.
Recall the definitions of the angles $\varphi_0$, $\varphi_1$, and $\varphi_2$ given by~\eqref{eq:inangle},  \eqref{eq:normal-angles-transformed-1}, and~\eqref{eq:normal-angles-transformed-2}, respectively.  

\begin{proposition}[Boundary estimate]
\label{prop:boundary-estimate}
Suppose that hypotheses~\aref{hyp:irreducible}, \aref{hyp:partial_homogeneity}, \aref{hyp:moments},
and~\aref{hyp:full_rank} hold.
Let $\alpha >0$, $\beta_1,\beta_2 \in(-\pi/2,\pi/2)$, and define $\beta$ as at~\eqref{eq:beta-def}. Suppose that, for $\nu > 1$ the constant in~\aref{hyp:moments}, $1 - \nu  < \alpha \beta < \nu$.
 Then, for each $k \in \{1,2\}$
and all $N \in \N$,
there exist $\lambda_k^N : \tRq_k \to (0,\infty)$ and $\veps^N_k : \tRq_k \to \R$ for which, 
\begin{align}
  \label{eq:boundary-increment}
 &  \Exp \bigl[ h^\alpha (\tZ^N_{n+1}) - h^\alpha (\tZ^N_n) \bigmid \tZ^N_n = z \bigr]  \nonumber\\
& {} \quad {}   = \alpha \beta \|z\|^{\beta-1} h^{\alpha-1} (z)  \lambda^N_k (z) \left[ \sin (\beta_k - \varphi_k) + \veps^N_k ( z) \right],  \text{ for all } z \in \tRq_k,
 \end{align}
 where (i)~for all $\veps >0$ there exist~$N \in \N$ and $r \in \RP$ such that $| \veps^N_k (z) | \leq \veps$
for all $\| z \| \geq r$; and (ii) for each $N \in \N$, $\inf_{z \in \tRq_k} \lambda^N_k (z) >0$
and $\sup_{z \in \tRq_k} \lambda^N_k (z) < \infty$.
\end{proposition}

The following consequence of Proposition~\ref{prop:boundary-estimate}
furnishes, once again, quantitative sub- and supermartingale estimates.
Denote $\sign x := \Ind{x>0}-\Ind{x<0}$.

\begin{corollary}
\label{cor:boundary-signs}
Suppose that hypotheses~\aref{hyp:irreducible}, \aref{hyp:partial_homogeneity}, \aref{hyp:moments},
and~\aref{hyp:full_rank} hold. Given $\nu > 1$, the constant in~\aref{hyp:moments}, suppose that  $\chi$ defined by~\eqref{eq:chi-def} and $\alpha > 0$
satisfy $\chi \neq 0$ and 
$1 - \nu  < \alpha \chi < \nu$.
Choose $\veps \in (0, | \chi |)$ for which $\alpha \veps < \nu - \alpha \chi$
and $\alpha \veps < \nu - 1+ \alpha \chi$.
Then there exist $\beta_1, \beta_2 \in (-\pi/2,\pi/2)$, $N \in \N$, $r \in \RP$, and $c > 0$ such that
$\beta$ given by~\eqref{eq:beta-def} satisfies $\sign \beta = \sign \chi$, 
 $| \chi| < | \beta| < | \chi | + \veps$, 
and
    \begin{equation}
  \label{eq:boundary-control-plus}
   \Exp \bigl[ h^\alpha (\tZ^N_{n+1}) - h^\alpha (\tZ^N_n) \bigmid \tZ^N_n = z \bigr] 
   \geq c \| z \|^{\alpha\beta -1}, \text{ for all } z \in \tRqb  \setminus B_r.
   \end{equation}
   Moreover,
   there exist (different from above)   $\beta_1, \beta_2 \in (-\pi/2,\pi/2)$, $N \in \N$, $r \in \RP$, and $c > 0$ such that $\sign \beta = \sign \chi$, $| \chi | - \veps < | \beta | < | \chi |$, 
and
    \begin{equation}
\label{eq:boundary-control-minus}
   \Exp \bigl[ h^\alpha (\tZ^N_{n+1}) - h^\alpha (\tZ^N_n) \bigmid \tZ^N_n = z \bigr] 
   \leq -c \| z \|^{\alpha\beta -1}, \text{ for all } z \in \tRqb \setminus B_r.
   \end{equation}
\end{corollary}

\begin{proof}[Proof of Proposition~\ref{prop:boundary-estimate}.]
For ease of notation, set 
$\Delta_N := \tZ^N_{1} - \tZ^N_0$ and $z' := \Ts^{-1} (z)$,
and note that~$\Delta_N = \Ts ( Z_N) - \Ts (Z_0)$ when $Z_0 \in \Rqb$ by~\eqref{eq:time-embedd}--\eqref{eq:time-change-process}.
Then, with the definition of~$d^N_k (z)$ from~\eqref{eq:n-step-drift-quadrant} and setting $\lambda_k^N (z) := \Exp_z \sum_{\ell=0}^{N-1} \Ind { Z_\ell \in \Rq_k}$,
we can write
\begin{equation}
    \label{eq:boundary-drift}
\Exp [ \Delta_N \mid \tZ^N_{0} = z ] = \Ts \Exp [ Z_{N} - Z_0 \mid \tZ_0 = z ]
= \lambda_k^N(z') \cdot \Ts\, d^N_k ( z' ) , \text{ for all } z \in\tRq_k.
\end{equation}
By hypothesis~\aref{hyp:moments},
there exist constants~$B < \infty$ and $\nu >1$ for which $\Exp [ \| \tZ_{n+1} - \tZ_n \|^\nu \mid \tZ_n =z ] \leq B$, for all $z \in \tRq$ and all $n \in \ZP$ and hence, by the simple inequality $( \sum_{i=1}^N |s_i| )^\nu \leq N^\nu \max_{1 \leq i \leq N} |s_i|^\nu \leq N^{\nu} \sum_{i=1}^N | s_i|^\nu$, we have
\begin{equation}
    \label{eq:Delta-N-moment-bound}
 \sup_{z \in \tRq}   \Exp [ \| \Delta_N \|^\nu \mid \tZ^N_0 = z ] \leq B_N ,
\end{equation}
for some constant $B_N < \infty$ that depends on $N$ and $\nu$.
Analogously to~\eqref{eq:E-delta}, define
\begin{equation}
    \label{eq:E-N-delta}
    E_{N,\delta} (z) := \bigl\{ \| \Delta_N \| \leq (1 + \| z \| )^{1-\delta} \bigr\}. \end{equation}
Fix $\delta \in (0,\frac{\nu-1-(\alpha\beta)^-}{\nu})$. For $z \in \tRqb$, we then have
\begin{align}
\label{eq:conditional-increment-f-partition-boundary}
\Exp[ h^\alpha (\tZ^N_{n+1}) - h^\alpha (\tZ^N_{n}) \mid \tZ^N_{n} = z] 
& =
    \Exp \bigl[ \bigl( h^\alpha (z + \Delta_N) - h^\alpha(z) \bigr) \Indno{E_{N,\delta} (z)} \bigmid \tZ_0 = z \bigr] \nonumber\\
    & {} \quad {} + \Exp \bigl[ \bigl( h^\alpha (z + \Delta_N) - h^\alpha(z) \bigr) \Indno{E_{N,\delta}^\rc (z)} \bigmid \tZ_0 = z \bigr].
\end{align}
The first term on the right-hand side of~\eqref{eq:conditional-increment-f-partition-boundary}
we will estimate using a Taylor expansion;
the stabilization property of the drift in~\eqref{eq:boundary-drift}
will allow us to fix $N$ (large). For the second term on the right-hand side of~\eqref{eq:conditional-increment-f-partition-boundary} we will use
Lemma~\ref{lem:big-jump}, which is applicable (for fixed $N$) because of the bound~\eqref{eq:Delta-N-moment-bound}.
Indeed, similarly to~\eqref{eq:big-jump-bound-minus-2}--\eqref{eq:h-big-jump},
 we obtain that
the second term on the right-hand side of~\eqref{eq:conditional-increment-f-partition-boundary} 
satisfies
\begin{equation}
    \label{eq:boundary-big-jump}
    \Exp \bigl[ \bigl| h^\alpha (z + \Delta_N) - h^\alpha(z) \bigr| \Indno{E_{N,\delta}^\rc (z)} \bigmid \tZ_0 = z \bigr] =o (\| z \|^{\alpha \beta -1} ), \text{ as } \| z \| \to \infty,
\end{equation}
provided that $1 - \nu < \alpha \beta < \nu$ and $\delta < \frac{\nu-1-(\alpha\beta)^-}{\nu}$.

We apply the first-order Taylor expansion~\eqref{eq:taylor-first-order}
with $f = h^\alpha$ and $y=\Delta_N$ to obtain
\begin{align}
\label{eq:first-order-taylor-h}
    &  \bigl(  h^\alpha (z + \Delta_N) - h^\alpha (z) \bigr) \Indno{ E_{N,\delta} (z) } \nonumber\\
    & {} \qquad {} = \alpha h^{\alpha-1} ( z ) \Delta_N^\tra \nabla h (z ) \Indno{ E_{N,\delta} (z) } + \fR^{h^\alpha}_2 (z, \Delta_N ) \Indno{ E_{N,\delta} (z) },
\end{align}
where, by~\eqref{eq:taylor-remainder} and the upper
bounds
on $h$ and its derivatives in~\eqref{eq:growth-factor} and~\eqref{eq:growth-bound-derivatives},
\[ 
\Exp \Bigl[  \bigl| \fR^{h^\alpha}_2 (z, \Delta_N) \bigr| \Indno{ E_{N,\delta} (z) } \Bigmid \tZ_0 = z \Bigr] \leq C_\delta \| z \|^{\alpha \beta -1 -\delta}\Exp [ \| \Delta_N \| \mid \tZ_0 = z ],
\]
where $ \Exp [ \| \Delta_N \| \mid \tZ_0 = z ]$ is uniformly bounded, since~\aref{hyp:moments} holds for $\nu >1$.
Hence we obtain, for the first term on the right-hand side of~\eqref{eq:conditional-increment-f-partition-boundary}, for all $z \in \tRq_k$,
\begin{align}
\label{eq:boundary-increment-main-term}
    &  \Exp \bigl[ \bigl( h^\alpha (z + \Delta_N) - h^\alpha(z) \bigr) \Indno{E_{N,\delta} (z)} \bigmid \tZ_0 = z \bigr]  \nonumber\\
     & {} \quad {} = \alpha h^{\alpha-1} ( z ) \Exp \bigl[ \bigl( \Delta_N^\tra \nabla h (z ) \bigr) \Indno{E_{N,\delta} (z)}  \bigmid \tZ_0 = z \bigr]      + o ( \| z\|^{\alpha \beta -1} ) ,
     \text{ as } \| z \| \to \infty. \end{align}
Moreover, from Lemma~\ref{lem:big-jump} and~\eqref{eq:h-derivatives}, we have that, as $\|z\| \to \infty$,
\[ \Exp \bigl[ \bigl| \Delta_N^\tra \nabla h (z ) \bigr| \Indno{E^\rc_{N,\delta} (z)}  \bigmid \tZ_0 = z \bigr] 
\leq \| \nabla h (z) \| \Exp \bigl[ \| \Delta_N \| \Indno{E^\rc_{N,\delta} (z)}  \bigmid \tZ_0 = z \bigr]
= o ( \| z \|^{\beta -1} ).
\]
Then 
using~\eqref{eq:conditional-increment-f-partition-boundary},
together with~\eqref{eq:boundary-big-jump}, \eqref{eq:boundary-increment-main-term}, and~\eqref{eq:boundary-drift}, we obtain, for all $z \in \tRq_k$,
     \begin{align}
    \label{eq:boundary-increment-main-term-2}
    &  \Exp \bigl[  h^\alpha (z + \Delta_N) - h^\alpha(z)   \bigmid \tZ_0 = z \bigr]  \nonumber\\      
& {} \quad {}  = \alpha \lambda_k^N (z')  h^{\alpha-1} ( z ) \bigl( \Ts\, d^N_k (z') \bigr)^\tra \nabla h ( z )    + o ( \| z\|^{\alpha \beta -1} ), \text{ as } \| z \| \to \infty.
\end{align}

We next use the stabilization result of \S\ref{sec:stabilization} to estimate  the term on the right-hand side of~\eqref{eq:boundary-increment-main-term-2} involving 
the inner product of the $N$-step drift 
with the gradient $\nabla h$. 
Let $v_1$ and $v_2$ be as defined in \eqref{eq:t_drift},
and the corresponding angles $\varphi_1$ and $\varphi_2$ relative to the 
appropriate normal vectors be as given by~\eqref{eq:normal-angles-transformed-1}
and~\eqref{eq:normal-angles-transformed-2}.
Recalling the computations for $\nabla h$ from~\eqref{eq:h-derivatives} and~\eqref{eq:norm-grad}, 
we obtain
from~\eqref{eq:boundary-drift-direction} that
\begin{align}
\label{eq:v1-gradient}
\frac{v_1^\tra \nabla h(z)}{\beta \| z \|^{\beta -1}} & = - \sin \bigl( (\beta-1) \theta -\beta_1 + \varphi_1 \bigr) , \\
\label{eq:v2-gradient}
\frac{v_2^\tra \nabla h(z)}{\beta \| z \|^{\beta -1}} & =   \sin \bigl( (\beta-1) \theta -\beta_1 + \varphi_0 - \varphi_2 \bigr) ,\end{align}
for all $z \in \R^2 \setminus \{ \0\}$.
For $z \in \tRq$, represented in polar coordinates by $(r,\theta)$,  we have 
\begin{align}
\label{eq:theta-boundary-1}
\theta & \to 0 \text{ as } \| z \| \to \infty \text{ with } z \in \tRqbx; \\
\label{eq:theta-boundary-2}
\theta & \to \varphi_0 \text{ as } \| z \| \to \infty \text{ with } z \in \tRqby.
\end{align} 
From~\eqref{eq:v1-gradient} and~\eqref{eq:v2-gradient} together with~\eqref{eq:theta-boundary-1} and~\eqref{eq:theta-boundary-2}, respectively, it follows that
\begin{align}
\label{eq:boundary-derivative-i}
\lim_{\| z\| \to \infty,\, z \in \tRq_k}&\frac{v_k^\tra \nabla h(   z)}{\beta \| z\|^{\beta - 1}}
= \sin(\beta_k - \varphi_k), \text{ for }  k \in \{1,2\}, 
\end{align}
where for the case $k=2$ we have used the fact that $\beta \varphi_0 - \beta_1 = \beta_2$, by~\eqref{eq:beta-def}.

Let $\veps >0$. By Proposition~\ref{prop:stabil}, we can choose $N \in \N$ large enough so that
$\| d^N_k (z') - \bmu_k \| < \| \Ts \bmu_k \| \veps$ for all $z' \in \Rq_k$ with $\| z' \| > 2RN$.
Then, by~\eqref{eq:t_drift},
\[ \bigl| \bigl( \Ts\, d^N_k (z') \bigr)^\tra \nabla h ( z ) - v_k^\tra \nabla h ( z ) \bigr| \leq  \veps, \text{ for all } z' \in \Rq_k \text{ with } \| z' \| > 2RN.
\]
Together with~\eqref{eq:boundary-derivative-i}, this means that we can choose large enough $r > 2RN$ so that
\[ \bigl| \bigl( \Ts\, d^N_k (z') \bigr)^\tra \nabla h (z ) - \beta \|  z \|^{\beta -1} \sin (\beta_k - \varphi_k ) \bigr| \leq  \| z \|^{\beta -1} \veps, \text{ for all } z \in \tRq_k \text{ with } \| z \| > r. 
\]
Using this in~\eqref{eq:boundary-increment-main-term-2} yields~\eqref{eq:boundary-increment}
together with the property~(i) for $\veps_k^N$ claimed in the proposition.
For fixed $N$,
we can choose $r$ large enough so that for all   $z \in \tRqbx$ with $\| z \| > r$, the term $\lambda_1^N (z)$
depends only on $e_2^\tra z$, which takes values in the finite set $\I{R}$. Thus  
we get the properties claimed in~(ii) for $k=1$; the case $k=2$ is similar.
\end{proof}

\begin{proof}[Proof of Corollary~\ref{cor:boundary-signs}]
Suppose that  $\chi \neq 0$, $\alpha >0$ and $\nu>1$ satisfy 
$1 - \nu  < \alpha \chi < \nu$. 
First consider the case where $\chi > 0$.
For $\veps \in (0,\chi)$ as specified in the corollary, 
pick $\beta_1,\beta_2 \in (-\pi/2,\pi/2)$ such that  for each~$k\in\{1,2\}$, 
\begin{equation}
\label{eq:pick_beta_k_1}
\varphi_k - \frac{\veps}{2}\varphi_0 < \beta_k <\varphi_k.
\end{equation} With this choice, by~\eqref{eq:chi-def} and~\eqref{eq:beta-def}, 
it holds that $0 < \chi -\veps < \beta < \chi$ (recall that $\varphi_0 \in (0,\pi)$ is defined at~\eqref{eq:inangle}). In particular, $\sign \beta = \sign \chi$.
Moreover, $0 < \alpha \beta < \alpha \chi < \nu$,
so that the hypotheses of Proposition~\ref{prop:boundary-estimate} are satisfied.
 By statement~(i) in Proposition~\ref{prop:boundary-estimate}, we can choose $N \in \N$ and $r \in \RP$
 such that, for each $k \in \{1,2\}$, 
 $\sin (\beta_k - \varphi_k ) + \veps_k^N (z) < -\veps$
 for all $z \in \tRqb  \setminus B_r$. Since $\beta >0$,
 we verify~\eqref{eq:boundary-control-minus} from~\eqref{eq:boundary-increment}. 
 Similarly, if we pick for each~$k \in \{1,2\}$,
\begin{equation}\label{eq:pick_beta_k_2}
 \varphi_k  < \beta_k <\varphi_k +   \frac{\veps}{2}\varphi_0,
 \end{equation}
 then
 $0 < \chi < \beta < \chi + \veps$. 
 Again,  $\sign \beta = \sign \chi$, and now
  $0 < \alpha \beta < \alpha \chi + \alpha \veps  < \nu$, by choice of $\veps$.
By 
Proposition~\ref{prop:boundary-estimate}, we can choose $N \in \N$ and $r \in \RP$
 such that, for each $k \in \{1,2\}$, 
 $\sin (\beta_k - \varphi_k ) + \veps_k^N (z) > \veps$
 for all $z \in \tRqb  \setminus B_r$. Since $\beta >0$,
 we verify~\eqref{eq:boundary-control-plus} from~\eqref{eq:boundary-increment}. 

On the other hand, consider $\chi < 0$.
Pick $\beta_1,\beta_2 \in (-\pi/2,\pi/2)$ such that for each $k \in \{1,2\}$,
\begin{equation}
\label{eq:pick_beta_k_3}
    \varphi_k  < \beta_k <\varphi_k +  \frac{\veps}{2}\varphi_0.
\end{equation}
Now $\chi < \beta < \chi + \veps < 0$,
and $\alpha \beta > \alpha \chi > 1 - \nu$, so, again, the hypotheses of Proposition~\ref{prop:boundary-estimate} hold. By Proposition~\ref{prop:boundary-estimate}, we can choose $N \in \N$ and $r \in \RP$
such that, for each $k \in \{1,2\}$, 
$\sin (\beta_k - \varphi_k ) + \veps_k^N (z) > \veps$ for all $z \in \tRqb  \setminus B_r$. Since $\beta < 0$,
we verify~\eqref{eq:boundary-control-minus} from~\eqref{eq:boundary-increment}. 
Similarly, if we pick $\beta_1,\beta_2 \in(-\pi/2,\pi/2)$ such that for each~$k \in \{1,2\}$, 
\begin{equation}
\label{eq:pick_beta_k_4}
\varphi_k - \frac{\veps}{2}\varphi_0  < \beta_k <\varphi_k.
\end{equation}
Then $\chi - \veps < \beta < \chi < 0$, while
  $\alpha \beta > \alpha \chi - \alpha \veps > 1 - \nu$, by choice of $\veps$. 
 By Proposition~\ref{prop:boundary-estimate}, we can choose $N \in \N$ and $r \in \RP$
 such that, for~$k \in \{1,2\}$, 
 $\sin (\beta_k - \varphi_k ) + \veps_k^N (z) < -\veps$
 for all $z \in \tRqb \setminus B_r$. Since $\beta < 0$,
 we verify~\eqref{eq:boundary-control-plus} from~\eqref{eq:boundary-increment}. 
\end{proof}

Corollary~\ref{cor:boundary-signs-truncated} below gives an analogue
of Corollary~\ref{cor:boundary-signs} for the 
truncated Lyapunov function defined using $h_b$ from~\eqref{eq:h-truncated}.   Corollary~\ref{cor:boundary-signs-truncated}
is proved by a minor modification of the proof of Corollary~\ref{cor:boundary-signs}, in which, 
similarly to \eqref{eq:increment_spli_adjust}--\eqref{eq:adjust_upper_bound_large_increment} in the proof of Proposition~\ref{prop:interior-estimate-truncated}, one 
separates large and small increments.

\begin{corollary}
    \label{cor:boundary-signs-truncated}
    Suppose that~\aref{hyp:irreducible}, \aref{hyp:partial_homogeneity}, \aref{hyp:moments},
and~\aref{hyp:full_rank} hold. Let $\nu > 1$ be the constant in~\aref{hyp:moments}, suppose that $\chi$ defined by~\eqref{eq:chi-def} and $\alpha > 0$
satisfy $\chi \neq 0$ and 
$1 - \nu  < \alpha \chi < \nu$.
Choose $\veps \in (0, | \chi |)$ for which $\alpha \veps < \nu - \alpha \chi$
and $\alpha \veps < \nu - 1 + \alpha \chi$.
Then there exist $\beta_1, \beta_2 \in (-\pi/2,\pi/2)$, $N \in \N$, $r \in \RP$, and $c > 0$ such that
$\beta$ given by~\eqref{eq:beta-def} satisfies $\sign \beta = \sign \chi$, 
 $| \chi| < | \beta| < | \chi | + \veps$, 
and, for all $b \geq r$,
\begin{equation}
\label{eq:boundary-control-plus-truncated}
\Exp \bigl[ h_b^\alpha (\tZ^N_{n+1}) - h_b^\alpha (\tZ^N_n) \bigmid \tZ^N_n = z \bigr]    \geq c \| z \|^{\alpha\beta -1}, \text{ for all } z \in \tRqb \cap ( B_b \setminus B_r ).
   \end{equation}
      Moreover,
   there exist (different from above)   $\beta_1, \beta_2 \in (-\pi/2,\pi/2)$, $N \in \N$, $r \in \RP$, and $c > 0$ such that $\sign \beta = \sign \chi$, $| \chi | - \veps < | \beta | < | \chi |$, 
and, for all $b \geq r$,
    \begin{equation}
\label{eq:boundary-control-minus-truncated}
   \Exp \bigl[ h_b^\alpha (\tZ^N_{n+1}) - h_b^\alpha (\tZ^N_n) \bigmid \tZ^N_n = z \bigr] 
   \leq -c \| z \|^{\alpha\beta -1}, \text{ for all } z \in \tRqb\cap ( B_b \setminus B_r ).
   \end{equation}
\end{corollary}

\subsection{Lower bounds on the extent of an excursion}
\label{sec:excursion-bounds}

An important ingredient in establishing lower tail bounds on passage times is the following
lower tail bound on the extent of an excursion, i.e., the probability that the trajectory reaches
a large distance from the origin before returning to a bounded set. Such an estimate yields
a corresponding tail bound on the return time, since trajectories of the process are essentially diffusive over large scales.  The technical result we will apply is Theorem~\ref{thm:moments-lower-bound} below;
see~\S\ref{sec:foster-lyapunov} for further remarks and references.

\begin{proposition}
\label{prop:excursion-bound}
 Suppose that~\aref{hyp:irreducible},   \aref{hyp:partial_homogeneity}, \aref{hyp:moments},
and~\aref{hyp:full_rank} hold, that~$\nu$ from~\aref{hyp:moments}
satisfies $\nu > 2$, and that $\chi$ defined by~\eqref{eq:chi-def}
satisfies $0 < \chi < \nu$.
Let~$r > 0$. Then there exists $r_0 \in (r,\infty)$ such that,
for every $\veps >0$, there is a $c_0 >0$ (depending on $\veps$, $r$, and $\beta$) such that
\begin{equation}
\label{lem:submartingale-bound-xi}
\Pr_z \Bigl( \max_{0 \leq n \leq \tau_Z (r) } \| Z_n \| \geq s \Bigr) \geq c_0 s^{-\chi-\veps},
\text{ for all } z \in \Rq \setminus B_{r_0} \text{ and all } s \geq 1.
\end{equation}
\end{proposition}
\begin{proof}
Suppose that $0<\chi<\nu$. Let $\alpha \in (1, \nu/\chi)$, and then pick $\veps \in (0, \frac{\nu}{\alpha} - \chi)$;
such choices of $\veps$ and $\alpha$ are possible since we assumed that $\chi < \nu$.
Corollary~\ref{cor:boundary-signs-truncated} shows that there exist
 $\beta_1, \beta_2 \in (-\pi/2,\pi/2)$, $N \in \N$, and $r_1 \in \RP$ such that
$\beta$ given by~\eqref{eq:beta-def} satisfies $\chi < \beta < \chi + \veps < \nu/\alpha$,
and, for all $b \geq r_1$,
\begin{equation}
\label{eq:boundary-control-plus-truncated-sign}
\Exp \bigl[ h_b^\alpha (\tZ^N_{n+1}) - h_b^\alpha (\tZ^N_n) \bigmid \tZ^N_n = z \bigr]    \geq 0, \text{ for all } z \in \tRqb \cap ( B_b \setminus B_{r_1} ).
   \end{equation}
Moreover, Proposition~\ref{prop:interior-estimate-truncated}, together with the fact that $\alpha > 1$ 
and $0 < \alpha \beta < \nu \leq \nuint$, shows that there exists $r_2 \in \RP$ 
such that, for all $b \geq r_2$,
\begin{equation}
\label{eq:interior-control-plus-truncated-sign}
\Exp \bigl[ h_b^\alpha (\tZ^N_{n+1}) - h_b^\alpha (\tZ^N_n) \bigmid \tZ^N_n = z \bigr]    \geq 0, \text{ for all } z \in \tRqint \cap ( B_b \setminus B_{r_2} ).
   \end{equation}
Combining~\eqref{eq:boundary-control-plus-truncated-sign} and~\eqref{eq:interior-control-plus-truncated-sign}, we conclude that for every $\alpha \in (1, \nu/\chi)$, there exist $N \in \N$ and $r_3 > (R\sqrt{2}) \vee r_1 \vee r_2$ such that, for all $b \geq r \geq r_3$, 
\begin{equation}
\label{eq:submartingale-control-plus-truncated}
\Exp \bigl[ h_b^\alpha (\tZ^N_{n+1}) - h_b^\alpha (\tZ^N_n) \bigmid \tZ^N_n = z \bigr]    \geq 0, \text{ for all } z \in \tRq \cap ( B_b \setminus B_{r} ).
   \end{equation}
   For simplicity of notation, write $\xi_n := h_b^\alpha (\tZ^N_{n})$. 
Define $\sigma_b := \inf \{ n \in \ZP : \| \tZ^N_n \| \geq b \}$
and $\tau_r := \inf \{ n \in \ZP : \| \tZ^N_n \| \leq r \} $.
From the definition of $h_b$ in~\eqref{eq:h-truncated} together with~\eqref{eq:submartingale-control-plus-truncated}, for every $b \geq r$
 with $r \geq r_3$, it holds that 
 $\xi_{n \wedge \tau_r \wedge \sigma_b}$
 is a non-negative submartingale, uniformly bounded by $(2 b)^{\alpha \beta}$.
Hence an application of the optional stopping theorem (e.g.~\cite[p.~33]{Bluebook}) at the (a.s.-finite)
stopping time $\tau_r \wedge \sigma_b$
shows that, for $z \in \tRq \cap ( B_b \setminus B_{r} )$, 
\begin{align*}
h^\alpha (z) = \Exp[ \xi_0 \mid \tZ^N_0 = z] & \leq \Exp [ \xi_{\tau_r \wedge \sigma_b} \mid \tZ^N_0 = z ] \\
& \leq r^{\alpha \beta}  \Pr  ( \tau_r < \sigma_b  \mid \tZ^N_0 = z )
+ (2 b)^{\alpha \beta} \Pr ( \sigma_b < \tau_r   \mid \tZ^N_0 = z ) ,
\end{align*}
using the fact that $| h(z) | \leq \| z \|^\beta$ so $\sup_{z \in B_r} h^\alpha (z) \leq r^{\alpha \beta}$. Thus, from the lower bound in~\eqref{eq:growth-factor} and the fact that $\alpha >1$, we obtain, for all $b \geq r \geq r_3$,
\[ \Pr ( \sigma_b < \tau_r  \mid \tZ^N_0 = z  ) \geq \frac{ \veps_0 \| z \|^{\alpha \beta} - r^{\alpha \beta}}{(2b)^{\alpha \beta} - r^{\alpha \beta}} , \text{ for all }  z \in \tRq \setminus  B_{r} . \]
There exists $r_0 >0$ (depending on $r$, $\Sigma$, $\beta$, and $\veps_0$, but not $\veps$)
such that $\| z \| \geq r_0$ implies (i) $\| \Ts (z) \| \geq r$, and (ii) $\veps_0 \| \Ts(z) \|^{\alpha \beta} - r^{\alpha \beta} \geq 1$, say; hence for all $b \geq r$,
\[ 
\Pr_z \Bigl( \max_{0 \leq n \leq \tau_r} \| \tZ^N_n \| \geq b  \Bigr) 
\geq
\Pr_z( \sigma_b < \tau_r  ) \geq \frac{c}{b^{\alpha\beta}} , \text{ for all } z \in \Rq \setminus  B_{r_0},
\]
 where $c >0$ depends on $\alpha, \beta$ (and hence $\veps$) but not $r$ or $b$. 
 A translation from $\tZ^N$ back to $Z$ (similarly to Lemma \ref{lem:time-change})  yields the statement in the lemma, noting that $\alpha \beta < \alpha \chi + \alpha \veps$ can be chosen to be arbitrarily close to $\chi$ by choice of $\alpha >1$ and then $\veps >0$.
\end{proof}

\subsection{Completing the proofs}
\label{sec:proof-of-theorem}

\begin{proof}[Proof of Theorem~\ref{thm:main-theorem}.]
Suppose that $\chi$ as defined at~\eqref{eq:chi-def} satisfies $\chi \neq 0$,
and let $\nu >1$ be the constant in hypothesis~\aref{hyp:moments}.
Let $\alpha \in (0,1)$ with $\alpha \chi < \nu$.
Combining Proposition~\ref{prop:interior-estimate}
and Corollary~\ref{cor:boundary-signs}, 
we deduce that for every $\veps \in (0, | \chi| )$ sufficiently small, there exist $N \in \N$, $r \in \RP$, $c >0$, and
$\beta_1, \beta_2 \in (-\pi/2,\pi/2)$
such that $\beta$ given by~\eqref{eq:beta-def}
satisfies $\sign \beta = \sign \chi$ and $| \chi | - \veps  < | \beta | < | \chi |$, and there holds the supermartingale estimate
\begin{equation}
    \label{eq:supermartingale}
    \Exp \bigl[ h^\alpha (\tZ^N_{n+1}) - h^\alpha (\tZ^N_n) \bigmid \tZ^N_n = z \bigr] 
   \leq -c \| z \|^{\alpha\beta -2}, \text{ for all } z \in \tRq \setminus B_r.
\end{equation}
In particular, if $\chi >0$ then $\beta >0$, and so $h (z) > 0$ for all $z \in \tRq \setminus \{ \0\}$ and
\[
  \lim_{r \to \infty}\inf \bigl\{ h(z) \colon
   z \in \tRq \setminus B_r \bigr\} = \infty,
\]
and then an application of Theorem~\ref{thm:recurrence}\ref{thm:recurrence-i} implies that $\tZ^N_n$
is recurrent,
and hence that $Z_n$ is recurrent, by Lemma~\ref{lem:time-change}. On the other hand, if $\chi <0$ then $\beta <0$,  the $h$
 for which~\eqref{eq:supermartingale}
holds satisfies $h (z) > 0$ for all $z \in \tRq \setminus \{ \0\}$, and
\[
  \lim_{r \to \infty}\inf \bigl\{ h(z) \colon
   z \in \tRq \setminus B_r \bigr\} = 0.
\]
Then Theorem~\ref{thm:recurrence}\ref{thm:recurrence-ii} implies that $\tZ^N_n$
is transient,
and hence  $Z_n$ is transient, by Lemma~\ref{lem:time-change}.
This completes the proof of part~\ref{thm:main-theorem-i} of the theorem.

For part~\ref{thm:main-theorem-ii}, we again use the supermartingale estimate provided by~\eqref{eq:supermartingale}.
Suppose $\chi > 0$, so that $\beta >0$.
From Lemma~\ref{lem:growth-factor}, we get from~\eqref{eq:supermartingale} that
\[
    \Exp \bigl[ h^\alpha (\tZ^N_{n+1}) - h^\alpha (\tZ^N_n) \bigmid \tZ^N_n = z \bigr] 
   \leq -c \bigl( h^\alpha (z) \bigr)^{1- \frac{2}{\alpha \beta}}, \text{ for all } z \in \tRq \setminus B_r.
\]
An application of Theorem~\ref{thm:moments-upper-bound}
with $\Lambda = B_r$, $f = h^\alpha$, and ${\eta} = 1- \frac{2}{\alpha \beta} < 1$ shows that
$\Exp [ \tau_Z(r)^\gamma ] < \infty$
for all~$\gamma < \frac{\alpha \beta}{2}$.
Here $\alpha, \beta$ were arbitrary subject to the constraints $\alpha \in (0,  1 \wedge \frac{\nu}{\chi})$
and $\beta \in (0,\chi)$, so we may conclude that $\Exp [ \tau_Z(r)^\gamma ] < \infty$
for all $\gamma < \frac{ \chi \wedge \nu}{2}$. Markov's inequality then completes the proof of
part~\ref{thm:main-theorem-ii} of the theorem.

Finally, we prove part~\ref{thm:main-theorem-iii}.
Suppose that $\chi >0$, and that~\aref{hyp:moments} holds with $\nu > 2 \vee \chi $. 
The $\alpha = 1/\beta$ case of 
Proposition~\ref{prop:boundary-estimate} (and an argument
similar to that in the proof of Corollary~\ref{cor:boundary-signs})
shows that we
may choose $\beta_1,\beta_2 \in (-\pi/2,\pi/2)$ such that $\beta > 0$ and, for all $r >0$ large enough, 
\begin{equation}
\label{eq:submartingale-boundary}
\Exp \bigl[   h^{1/\beta} (\tZ^N_{n+1}) - h^{1/\beta} ( \tZ^N_n)  \bigmid \tZ^N_n = z \bigr]
\geq 0, \text{ for all } z \in \tRqb \setminus B_r .
\end{equation}
We have also  from Proposition~\ref{prop:interior-estimate} that
for some $c_\beta \in \R$ (one may choose $c_\beta >0$ if $\beta < 1$),
\begin{equation}
\label{eq:submartingale-interior}
\Exp \bigl[   h^{1/\beta} (\tZ^N_{n+1}) - h^{1/\beta} ( \tZ^N_n)  \bigmid \tZ^N_n = z \bigr]
\geq \frac{c_\beta}{\| z \|}, \text{ for all } z \in \tRqint \setminus B_r .
\end{equation}
The estimates~\eqref{eq:submartingale-boundary} and~\eqref{eq:submartingale-interior} together with
Lemma~\ref{lem:growth-factor} imply that, for some $c >0$ (depending on $\beta$), and all $r > \sqrt{2}R$ large enough,
\begin{equation}
\label{eq:lamperti-lower-bound}
\Exp \bigl[   h^{1/\beta} (\tZ^N_{n+1}) - h^{1/\beta} ( \tZ^N_n)  \bigmid \tZ^N_n = z \bigr]
\geq - \frac{c}{h^{1/\beta} (z)}, \text{ for all } z \in \tRq \setminus B_r .
\end{equation}
Next, recall the definition of $E_{N,\delta}$ from~\eqref{eq:E-N-delta}. From~\eqref{eq:first-order-taylor-h}
together with~\eqref{eq:taylor-remainder} and the upper
bounds on $h$ and its derivatives in~\eqref{eq:growth-factor} and~\eqref{eq:growth-bound-derivatives},
we have that
\[
\bigl|  h^\alpha (z + \Delta_N) - h^\alpha (z) \bigr| \Indno{ E_{N,\delta} (z) } 
\leq C ( 1+  \| z \|)^{\alpha \beta - 1}  (1+ \| \Delta_N \|) , \]
for some constant $C < \infty$ and all $z \in \tRq$.
On the other hand, similarly to~\eqref{eq:big-jump-interior},
\[ \bigl| h^\alpha (z + \Delta_N) - h^\alpha(z) \bigr| \Indno{E^\rc_{N,\delta} (z)}
\leq C    \| \Delta_N \|^{(\alpha \beta)^+/(1-\delta)}.
\]
Choose $\alpha = 1/\beta$. Then from the last two bounds, for a sufficiently small $\delta>0$ we obtain
\[ \Exp \Bigl[ \bigl(  h^{1/\beta} (\tZ^N_{n+1}) - h^{1/\beta} ( \tZ^N_n) \bigr)^2 \Bigmid \tZ^N_n = z \Bigr] \leq C  \Exp \bigl[ (1+ \| \Delta_N \|)^2 \bigmid \tZ^N_0 = z \bigr] ,\]
which is bounded uniformly for $z \in \tRq$, provided~\aref{hyp:moments} holds with $\nu \geq 2$.
This fact, together with~\eqref{eq:lamperti-lower-bound}, shows that $f = h^{1/\beta}$   satisfies the conditions of Theorem~\ref{thm:moments-lower-bound} with $\Lambda = B_r$, from which it follows that, for some $C \in (0,\infty)$,
\[ \Pr ( \ttau_Z(r) \geq n ) \geq \frac{1}{2} \Pr \Bigl( \max_{0 \leq m \leq \ttau_Z(r)} h^{1/\beta} ( \tZ^N_m ) \geq C n^{1/2} \Bigr) , \text{ for all } n \in \ZP.
\]
From Lemma~\ref{lem:growth-factor} and a comparison between hitting times for $\tZ^N$ and $Z$, as in Lemma~\ref{lem:time-change}, there is $C'>0$ for which
\[ \Pr ( \tau_Z(r) \geq n ) \geq \frac{1}{2} \Pr \Bigl( \max_{0 \leq m \leq \tau_Z(r')} \|  Z_n \| \geq C' n^{1/2} \Bigr) , \text{ for all } n \in \ZP,
\]
and then Proposition~\ref{prop:excursion-bound} (which is applicable since $\nu > 2 \vee \chi$) 
completes the proof.
\end{proof}

\section{Examples and applications}
\label{sec:examples}

\subsection{The left-continuous case}
\label{sec:petritis}

In this subsection, we consider the case where we impose the additional condition
\begin{equation}
\label{eq:nearest-neighbour}
    \pint (x,y) = 0 , \text{ unless } x \geq -1, y \geq -1;
\end{equation}
this \emph{left-continuity} assumption
can facilitate analysis (see e.g.~\cite{km14} for a recent example).
 When $R>1$, left-continuity is stronger than the bound~\eqref{eq:bounded-jumps} which is implied by partial homogeneity alone. Note that~\eqref{eq:nearest-neighbour} imposes a constraint only in $\Rqint$ (and not $\Rqb$). 

 The assumption~\eqref{eq:nearest-neighbour} leads to a significant simplification. Indeed, we will demonstrate that the probability measures $\pi_1, \pi_2$ on $\I{R}$ defined in  Proposition~\ref{prop:projection-chains}\ref{prop:projection-chains-ii} can   be computed by solving an explicit linear system.
 Indeed, in general the $\pi_k$ are defined through the invariant measure of $\ZP$-valued Markov chains, and can be represented through finite chains via embedding (see Remark~\ref{rem:embedded-chain}) but the transition matrices of the resulting embedded chains  are usually described only implicitly.
 With $q_1, q_2$ defined at~\eqref{eq:q1-def} and~\eqref{eq:q2-def}, define $\barq_1, \barq_2 : \I{R} \times \I{R} \to [0,1]$ by
 \[ \barq_k ( i, j) = \begin{cases} q_k (i, j) &\text{if } j \in \I{R-1}, \\
 \sum_{\ell \in \bI{R-1} } q_k (i, \ell) &\text{if } j = R-1.
 \end{cases}
 \]
 Then $\barq_1$ is the transition function of an irreducible Markov chain on the finite state space $\I{R}$
 which depends only on finitely many values of the basic model data, namely $\sum_{x \in \Z} p_1 (i; (x,j))$
 with $i, j \in \I{R}$; similarly for $\barq_2$.
 
 \begin{lemma}
     \label{lem:finite-system}
     Under the condition~\eqref{eq:nearest-neighbour},
     for each $k \in \{1,2\}$, the distribution $\pi_k$ on $\I{R}$ defined in 
      Proposition~\ref{prop:projection-chains}\ref{prop:projection-chains-ii} is the unique stationary distribution corresponding to $\barq_k$.
     \end{lemma}

We give an algebraic proof of Lemma~\ref{lem:finite-system}
at the end of this subsection; 
a probabilistic explanation goes via the embedded chain described in  Remark~\ref{rem:embedded-chain}.

Under the additional assumption~\eqref{eq:nearest-neighbour},
as well as a couple of further inessential assumptions, including aperiodicity, 
Theorem~\ref{thm:main-theorem} was obtained by Menshikov \& Petritis (MP) in~\cite{MenPet04}.
The model of~\cite{MenPet04} is formulated a little differently, however, and some translation is needed, as we explain. MP consider a random walk $W = (W_n, n \in \ZP)$ on the quadrant $\ZP^2$
in which the walker also carries an \emph{internal state}
with values in the finite set $A := \{0,\ldots,L\}$, $L \in \ZP$;
the internal state is $0$ unless the walker is at the boundary,
in which case its internal state can be ``excited''.
The model can be represented as a Markov chain on state space 
$\bbW := \bbWint \cup \bbWbx \cup \bbWby \cup \bbWc$, where
\[ \bbWint :=  \N^2 \times \{ 0 \}, ~
\bbWbx := \N \times \{ 0 \} \times A, ~
\bbWby := \{ 0\} \times \N \times A, ~
\bbWc :=  \{ \0 \} \times A .\]
Thus $(w_1, w_2, \alpha ) \in \bbW$
is either of the form $(w_1, w_2, 0)$
for $w_1 w_2  \geq 1$ (in the interior) or
$(w_1, w_2, \alpha)$
for $\alpha \in A$ and $w_1 w_2 = 0$ (at the boundaries).
Partial homogeneity is assumed in the sense that, given $W_n = w$,
the law of $W_{n+1} - W_n$ depends only on which region $\bbWint, \bbWbx, \bbWby, \bbWc$
contains~$w$.
Furthermore, it is assumed in~\cite{MenPet04} that the increments of each of the two spatial coordinates take values in $\{-1\} \cup \ZP$; 
MP call this assumption ``lower boundedness'', and we will see that it
translates to the left-continuity assumption~\eqref{eq:nearest-neighbour}.

We show how, with a slight change of variables, the model of~\cite{MenPet04}
fits into our framework. We define the map $\Upsilon : \bbW \to \ZP^2$
by setting 
\begin{equation}
\label{eq:men-pet-map}
\Upsilon (w_1,w_2,\alpha) :=
\begin{cases}
(L+w_1,L+w_2) &\text{ if } w_1 w_2 \geq 1; \\
(L - \alpha,L+w_2) &\text{ if } w_1 =0, w_2 \geq 1; \\
(L+w_1 ,L - \alpha) &\text{ if } w_1 \geq 1, w_2 =0; \\
(L-\alpha, L-\alpha) &\text{ if } w_1 = w_2 = 0.
\end{cases}
\end{equation}
Then $Z = (Z_n, n \in \ZP)$ defined by $Z_n := \Upsilon (W_n)$ is a Markov chain on $\Rq := \ZP^2$,
and it satisfies the partial homogeneity condition~\aref{hyp:partial_homogeneity}
with $R = L$ and the identification of regions
$\Rqint = \Upsilon ( \bbWint )$, $\Rqbx = \Upsilon (\bbWbx)$,
$\Rqby = \Upsilon (\bbWby)$, and $\Upsilon (\bbWc) \subseteq \Rqc = \I{L} \times \I{L}$ (only the ``diagonal'' states in the corner are accessible). If $W$ is irreducible, then $\Upsilon (W)$ is irreducible only on $\Upsilon (\bbW) \subseteq \Rq$, rather than all of $\Rq$ (due to the restriction on corner states), but nevertheless, under very mild conditions, hypothesis~\aref{hyp:irreducible} is satisfied.

\begin{proof}[Proof of Lemma~\ref{lem:finite-system}.]
     Consider first $j \in \I{R-1}$. From~\eqref{eq:nearest-neighbour} it follows that
     $q_k (i,j) = 0$ if $i \in \bI{R}$ and~$i - j > 1$, and so
     \begin{align*}
         \sum_{i \in \I{R}} \pi_k (i) \barq_k ( i,j) & = \sum_{i \in \ZP} \pi_k (i) q_k (i, j)  = \pi_k (j) ,
     \end{align*}
     by the invariance of~$\pi_k$
     as at~\eqref{eq:invariant-pi-def}.
     Similarly, for $j = R -1 \in \I{R}$,
          \begin{align*}
         \sum_{i \in \I{R}} \pi_k (i) \barq_k ( i, R-1) & =  
          \sum_{i \in \I{R}} \pi_k (i) {\sum_{\ell \in \bI{R-1}}} q_k (i, \ell) \\
         & =  \sum_{i \in \I{R}}  { \pi_k (i)}\biggl[ 1 - \sum_{\ell \in \I{R-1}} q_k (i,\ell) \biggr] \\
         & = 1 - \sum_{\ell \in \I{R-1}}  \sum_{i \in \ZP}  \pi_k (i)  q_k (i,\ell) \\
         & = 1 - \pi_k ( \I{R-1} ) = \pi_k (R-1) ,
     \end{align*}
     as required.
 \end{proof}

\subsection{Multidimensional Lindley processes and related walks}
\label{sec:lindley}

In this section, we derive Theorem~\ref{thm:lindley} as
a consequence of our main result, Theorem~\ref{thm:main-theorem}. Indeed, the special nature of the reflections for the two processes $M$ and $L$ means that the result takes a particularly simple form, in which $\pi$ plays no apparent role; this is due to the following 
fact.

\begin{lemma}
    \label{lem:lindley-drifts}
    For both the Lindley random walk
    and the mirror-reflected walk,  
the boundary drifts
$\bmu_k$ have $\bmu_k \neq 0$, $k \in \{1,2\}$, and satisfy  $\bmu_1 = \| \bmu_1 \| e_2$ and $\bmu_2 = \| \bmu_2 \|e_1$.
\end{lemma}
\begin{proof}
First consider the mirror-reflected walk. 
Take $z = (x,y) \in \Rqbx$. We consider
\[
\mu_1 (y) = \Exp  | z + \zeta | - z . 
\]
Since $x \geq R$, $x + e_1^\tra \zeta >0$, a.s., so $e_1^\tra \mu_1 (y) =  \Exp e_1^\tra \zeta = 0$. On the other hand, suppose $r \in \{1,2,\ldots, R-1\}$ is such that $\Pr ( e_2^\tra \zeta = - r ) > 0$ (some such $r$ must exist, since $\Exp e_2^\tra \zeta =0$ and $\Sigma$ is positive definite). Then $r-1 \in \I{R}$ and 
$e_2^\tra \mu_1 ( r -1) > 0$.  
Thus $\bmu_1$ is a weighted sum of vectors all of which have $0$ component in the $e_1$ direction, non-negative components in the $e_2$ direction, and at least one strictly positive component in the $e_2$ direction. Hence $\bmu_1 = \| \bmu_1 \| e_2$. An analogous argument holds  for $\bmu_2$.

For the Lindley walk,   for $z = (x,y) \in \Rqbx$ one instead considers
\[
\mu_1 (y) = \Exp \bigl[ ( z + \zeta )^+ \bigr] - z , \]
to which a very similar argument applies.
\end{proof}

\begin{proof}[Proof of Theorem~\ref{thm:lindley}]
Suppose that~\aref{hyp:lindley-increments} holds, and that $Z$ is either the Lindley walk $L$ or the mirror-reflected walk $M$. For the constant $R$ from
the condition $\Pr ( e_k^\tra \zeta \geq -R) =1$ in~\aref{hyp:lindley-increments}, define the regions $\Rqint$, $\Rqbx$, $\Rqby$ as in \S\ref{sec:partial_homogeneity}.
For $z \in \Rqint$, we then have $\Pr ( z + \zeta \in \Rq ) = 1$, meaning that $\Pr ( Z_1 - Z_0 = y \mid Z_0 = z) = \Pr ( \zeta = y)$ for all $z \in \Rqint$, $y \in \Z^2$. 
In particular, it follows from the assumptions that $\Exp \zeta = \0$ and positive-definiteness of $\Sigma$ from~\aref{hyp:lindley-increments} that~\aref{hyp:zero_drift} and~\aref{hyp:full_rank} are satisfied.
Moreover, since both $\| | z+ \zeta| - z \| \leq \| \zeta \|$ and $\| (z+ \zeta)^+ - z\| \leq \| \zeta \|$ for all $z, \zeta \in \Z^2$, we have, for any $\nu >0$, 
\begin{equation}
    \label{eq:lindley-bound}
\Exp [ \| Z_1 - Z_0 \|^\nu \mid Z_0 = z] \leq \Exp [ \| \zeta \|^\nu ], \text{ for all } z \in \Rq.\end{equation}

The condition~\aref{hyp:irreducible} is a consequence of the 
assumed irreducibility
together with the fact that $\Exp \zeta = \0$ and positive-definiteness of $\Sigma$, which imply uniform ellipticity in $\Rqint$ (cf.~Remark~\ref{rem:moments}\ref{rem:moments-iii}).
We show that the partial homogeneity hypothesis~\aref{hyp:partial_homogeneity} is satisfied. To do so, consider the corresponding transition laws (in particular, started from $\Rqb$): see Figure~\ref{fig:full} for an illustration of the reflection mechanism in each case. One verifies that the mirror-reflected walk  satisfies~\aref{hyp:partial_homogeneity}, with $R$ the constant inherited from~\aref{hyp:lindley-increments},
and with $\pint, p_1, p_2$ defined for $z=(z_1,z_2) \in \Rq$ by
\begin{align*}
    \pint (z) & = \Pr ( \zeta = z ) ; \\
    p_1 (y ; z) & =   \Pr ( \zeta = z ) \Ind { y + z_2 \geq 0} +   \Pr ( \zeta = ( z_1 , - 2y - z_2) )  \Ind { y + z_2 > 0} ; \\
     p_2 (x ; z) & = \Pr ( \zeta = z ) \Ind { x + z_1 \geq 0} +   \Pr ( \zeta = ( -2x - z_1 , z_2) )  \Ind { x + z_1 > 0 } ;
\end{align*}
see the left-hand panel in Figure~\ref{fig:full}.
This shows that $Z = M$ satisfies~\aref{hyp:partial_homogeneity}, and then~\eqref{eq:lindley-bound} together with the assumption $\Exp [ \| \zeta \|^\nu ] < \infty$ from~\aref{hyp:lindley-increments} shows that~\aref{hyp:moments} holds. 

For the Lindley walk, the considerations are similar, but now we verify that~\aref{hyp:partial_homogeneity} holds with the identification of $p_1, p_2$ through 
\begin{align*}
    p_1 (y ; z) & =   \Pr ( \zeta = z ) \Ind { y + z_2 \geq 0} +   \sum_{w \in \Z \colon y + w < 0} \Pr ( \zeta = ( z_1 , w) )  \Ind { y + z_2 = 0} ; \\
     p_2 (x ; z) & = \Pr ( \zeta = z ) \Ind { x + z_1 \geq 0} +   \sum_{w \in \Z \colon x + w < 0}  \Pr ( \zeta = (w , z_2) )  \Ind { x + z_1 = 0 } ;
\end{align*}
see the right-hand panel in Figure~\ref{fig:full}.

We have thus verified all the hypotheses of Theorem~\ref{thm:main-theorem}.
Moreover, Lemma~\ref{lem:lindley-drifts} shows that, in both cases,
 the reflection angles are orthogonal, i.e.,
 $\bmu_1$ is in direction $e_2$ and $\bmu_2$ is in direction $e_1$. Thus 
 (see Example~\ref{ex:orthogonal}) it holds that $\varphi_1 = \varphi_2 = \varphi_0 - \frac{\pi}{2}$,
 and $\chi = 2 - \frac{\pi}{\arccos (-\rho)}$. Hence $\chi >0$ is equivalent to $\rho >0$, while $\chi <0$ is equivalent to $\chi <0$: thus Theorem~\ref{thm:main-theorem}\ref{thm:main-theorem-i} establishes the recurrence/transience classification given in Theorem~\ref{thm:lindley}. Finally, if $\rho >0$ (hence $\chi >0$) then Theorem~\ref{thm:main-theorem}\ref{thm:main-theorem-ii}--\ref{thm:main-theorem-iii}
 yields~\eqref{eq:tail-Lindley}, following the formulation in Remark~\ref{rem:moments}\ref{rem:moments-ii}.
\end{proof}

\begin{figure}[!h]
  \centering
  \begin{subfigure}[b]{6.5cm}
    \centering
    \begin{tikzpicture}
        \draw[->] (0,0) --(0,4){};
        \draw[->] (0,0) --(5,0){};
        \draw[dashed] (0,1) -- (5,1) {};
        \node at (-.5,1) {$y$};
        \draw (1,1) circle (8pt);
        \draw (2.8,1.7) rectangle ++ (16pt,16pt);
        \draw[<->] (.5,1) --(.5,2);
        \draw[<->] (1,2) --(3,2);
        \node[isosceles triangle,
        isosceles triangle apex angle=60,
        draw,
        rotate=-30,
        minimum size =1pt] (T1) at (3,-2) {};
        \node at (.75,1.6) {$z_2$};
        \node at (1.75,2.3) {$z_1$};
        \draw[<->] (1.5,1)--(1.5,-2);
        \node at (2.5,-1) {$-2y - z_2$};
    \end{tikzpicture}
    \caption{Mirror-reflected random walk.}
    \label{fig:reflected-jump-increment}
  \end{subfigure}
  \hfill
  \begin{subfigure}[b]{6.5 cm}
    \centering
    \begin{tikzpicture}
        \draw[->] (0,0) --(0,4){};
        \draw[->] (0,0) --(5,0){};
        \draw[dashed] (0,1) -- (5,1) {};
        \node at (-.5,1) {$y$};
        \draw (1,1) circle (8pt);
        \draw (2.7,-8 pt) rectangle ++ (16pt,16pt);
        \draw[->, dotted, line width = 0.6mm] (3,-1.5) --(3,-8pt);
        \draw[->] (1.15,.8) --(3,-2);
        \draw[->] (1.15,.8) --(3,-1.7);
    \end{tikzpicture}
    \caption{Lindley random walk.}
    \label{fig:Lindley-2}
  \end{subfigure}
 \caption{An illustration of partial homogeneity and the identification of the boundary transition function $p_1(y;z)$ for the increment started from $(x,y) \in \Rqbx$, for the mirror-reflected walk~$M$ (\emph{left}) and the Lindley walk~$L$ (\emph{right}). For~$M$, the probability of an increment of size $(z_1,z_2)$ is the sum of the probability of the direct jump from the circle to the square plus the probability of the reflected jump from the circle to the triangle. For~$L$, the probability to make an increment of $(z_1,-y)$ includes all increments $(z_1,w)$, with $w$ such that $w + y \leq 0$, the action of taking the positive part of $y+w$ represented by the dotted line.}
  \label{fig:full}  
\end{figure}
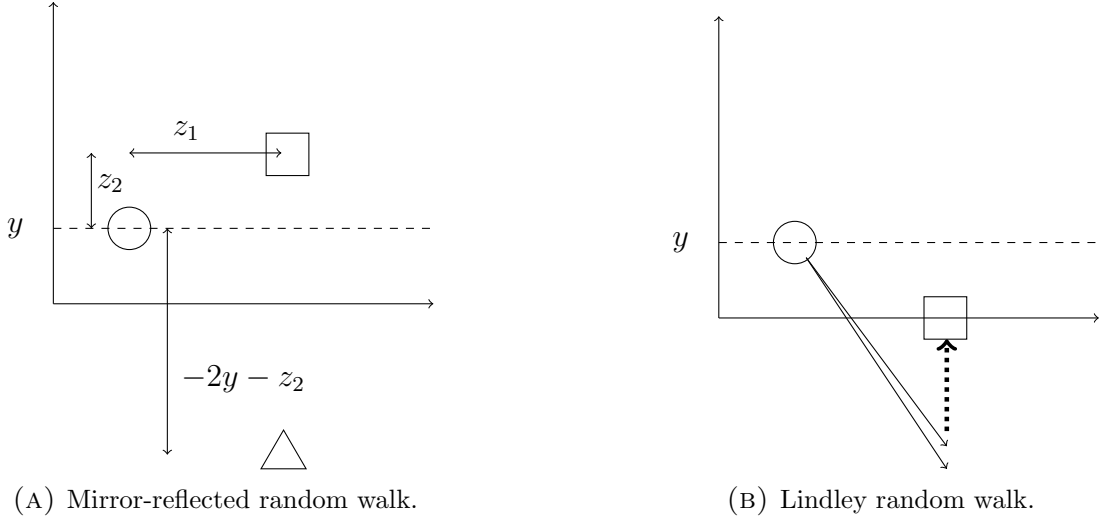

\appendix

\section{Foster--Lyapunov conditions}
\label{sec:foster-lyapunov}

In this section, $Z = (Z_n, n \in \ZP)$ will denote a Markov chain on a countable
state space $\bb{X}$,
and passage times to sets $A \subseteq \bb{X}$ will be denoted by $\tau_Z(A)$ as at~\eqref{eq:def-passage-time}.
For a function  $f: \bb{X} \to \RP$, we write $f \to \infty$
to mean that for every $r \in \RP$,  $\{z \in \bb{X} : f(z) \leq r\}$ is finite.
The following Foster--Lyapunov conditions for recurrence/transience
of countable Markov chains are standard, and can be found in e.g.~Theorems~2.5.2 and~2.5.8 of~\cite[\S 2.5]{Bluebook}. 

\begin{theorem}[Recurrence/transience]
\label{thm:recurrence}
 Let $(Z_n,n \in\ZP)$ be an irreducible Markov chain on a countable state space $\bb{X}$, and suppose that there exist a function
 $f: \bb{X} \to \RP$ and a
  set $\Lambda \subset \bb{X}$ such that
   \begin{equation}
   \label{eq:supermart-f}
   \Exp \bigl[ f(Z_{n+1})- f(Z_n) \mid Z_n = z \bigr] \leq 0, \text{ for all } z \in 
    \bb{X} \setminus \Lambda . \end{equation}
 Then the following apply.
   \begin{enumerate}[label=(\roman*)]
   \item 
   \label{thm:recurrence-i}
   If~$\Lambda$ is finite and $f \to \infty$, then $Z$ is recurrent.
   \item
   \label{thm:recurrence-ii}
   If~$\Lambda \neq \emptyset$ and there exists $z \in \bb{X} \setminus \Lambda$ with $f(z) < \inf_{w \in \Lambda} f(w)$, then $Z$ is transient.
   \end{enumerate}
\end{theorem}

Theorem~\ref{thm:moments-upper-bound} gives upper bounds 
on tails of passage times, and can be seen as an extension of Foster's criterion for integrability
of passage times; it is contained in a result of~\cite[\S 2.7]{Bluebook}, whose
origins go back to Lamperti~\cite{Lamp63} and Aspandiiarov \emph{et al.}~\cite{AspIasMen96}.

\begin{theorem}[Upper tail bounds]
\label{thm:moments-upper-bound}
 Let $(Z_n,n \in\ZP)$ be a Markov chain on $\bb{X}$.
 Suppose that there exist 
 $f: \bb{X} \to (0,\infty)$ with $f \to \infty$,  
 constants $\eta < 1$, $c >0$, and a
    set $\Lambda \subset \bb{X}$ such that
   \begin{equation}
   \label{eq:strong-supermart-f}
   \Exp \bigl[ f(Z_{n+1})- f(Z_n) \mid Z_n = z \bigr] \leq - c f (z)^\eta, \text{ for all } z
     \in \bb{X} \setminus \Lambda . \end{equation}
       Then for any $\gamma \in [0,\frac{1}{1-\eta})$, $\Exp [ \tau_Z(\Lambda)^\gamma ] < \infty$.
  \end{theorem}
\begin{proof}
Apply Corollary~2.7.3 of~\cite[pp.~70]{Bluebook} to the process $X_n = f(Z_n)$.
\end{proof}

Lower tail bounds are more delicate.
A general approach to lower bounds of passage times is developed in~\cite[\S\S2.7, 3.7, 3.8]{Bluebook}, and its origins can be traced in one form or another back to Doob and, primarily, Lamperti~\cite{Lamp63}.
 Theorem~\ref{thm:moments-lower-bound} below requires a Lyapunov function $f$ for which conditions~\eqref{eq:non-existence-1}--\eqref{eq:non-existence-2} hold, which can 
can be understood as requiring that $f(Z)$ can return from a distant level to near the origin
no faster than diffusively. 
If this is so, then a lower bound on the probability that an excursion of $f(Z)$ reaches level $n^{1/2}$ (say)
should provide a lower bound on the probability that the duration of the excursion is of order~$n$.
Theorem~\ref{thm:moments-lower-bound} makes this precise. 
In applications of Theorem~\ref{thm:moments-lower-bound}, one also needs
a lower bound on the probability of reaching a distant level during an excursion,
and this is often provided by an optional stopping argument: for example, if $f^\alpha (Z)$
is close to a martingale, then one would obtain a tail bound of the form  $\Pr ( \tau_Z(\Lambda) \geq n ) \gtrsim n^{-\alpha/2}$, roughly speaking (in our setting, \S\ref{sec:excursion-bounds} gives such a bound).

\begin{theorem}[Lower tail bounds]
\label{thm:moments-lower-bound}
 Let $(Z_n,n \in\ZP)$ be a Markov chain on $\bb{X}$.
Suppose that there exist 
 $f: \bb{X} \to (0,\infty)$ with $f \to \infty$,    constants $B \in \RP$, $c \in \R$, and a
    set $\Lambda \subset \bb{X}$ such that
   \begin{align}
   \label{eq:non-existence-1}
   \Exp \bigl[ ( f(Z_{n+1})- f(Z_n) )^2 \mid Z_n = z \bigr] & \leq  B, \text{ for all } z
     \in \bb{X} \setminus \Lambda ;\\
      \label{eq:non-existence-2}
      \Exp \bigl[  f(Z_{n+1})- f(Z_n)   \mid Z_n = z \bigr] & \geq  - \frac{c}{f(z)}, \text{ for all } z
     \in \bb{X} \setminus \Lambda .
     \end{align}
Then there exists $C \in (0,\infty)$ such that
\[ \Pr \bigl( \tau_Z(\Lambda) \geq n \bigr) \geq \frac{1}{2} \Pr \left( \max_{0 \leq m \leq \tau_Z(\Lambda) } f (Z_m) \geq C n^{1/2} \right) .\]
       \end{theorem}
\begin{proof}
Apply Lemma~2.7.7 of~\cite[pp.~76]{Bluebook} to the process $X_n = f(Z_n)$.
\end{proof}

\section*{Acknowledgements}
\addcontentsline{toc}{section}{Acknowledgements}

The authors are grateful for the constructive feedback from the anonymous referees. 
This work was supported by EPSRC grant EP/W00657X/1.
Revision of the paper was undertaken  during the programme ``Stochastic systems for anomalous diffusion'' (July--December 2024) hosted by the  Isaac Newton Institute, under EPSRC grant EP/Z000580/1.

\bibliographystyle{amsplain}

\end{document}